\documentclass[12pt, a4paper,reqno]{amsart}
\usepackage[a4paper, margin=3cm]{geometry}
\usepackage{amsmath}
\usepackage{amssymb,amsfonts}
\usepackage{tikz-cd}
\usepackage{tikz}

\usepackage{color}
\usepackage[colorlinks=true,linkcolor=blue,citecolor=blue,pdfpagelabels=false]{hyperref}
\usepackage{cleveref}
\usepackage{url}

\newtheorem{Theorem}{Theorem}[section]

\theoremstyle{definition}

\newtheorem{thm}{Theorem}[section]
\newtheorem{cor}[thm]{Corollary}
\newtheorem{lem}[thm]{Lemma}
\newtheorem{prop}[thm]{Proposition}
\newtheorem{defn}[thm]{Definition}

\newtheorem{rem}[thm]{Remark}
\numberwithin{equation}{section}

\def\C{\mathbb C}

\def\F{\mathbb F}

\def\N{\mathbb N}

\def\R{\mathbb R}

\def\Z{\mathbb Z}
\def\K{\mathbb K}

\newcommand{\CB}{\mathcal{B}}

\newcommand{\CH}{\mathcal{H}}

\newcommand{\CP}{\mathcal{P}}

\newcommand{\CW}{\mathcal{W}}

\def\be{\begin{equation}}
	\def\ee{\end{equation}}
\def\bt{\begin{Theorem}}
	\def\et{\end{Theorem}}
\def\bi{\begin{itemize}}
	\def\ei{\end{itemize}}
\def\bea{\begin{eqnarray}}
	\def\eea{\end{eqnarray}}
\def\beast{\begin{eqnarray*}}
	\def\eeast{\end{eqnarray*}}
\def\ben{\begin{enumerate}}
	\def\een{\end{enumerate}}

\def\bi{\bibitem}

\newcommand{\norm}[1]{\left\Vert#1\right\Vert}

\newcommand\inner[2]{\left\langle #1, #2 \right\rangle}

\renewcommand{\MR}[1]{} 

\newcommand{\abs}[1]{\left\vert#1\right\vert}

\allowdisplaybreaks[4]

\begin{document}
	\title[noncommutative ergodic theorem]{On noncommutative ergodic theorems for semigroup and free group actions}

	\author[P. Bikram]{Panchugopal Bikram}
	\author[D. Saha]{Diptesh Saha}

	
	\address{School of Mathematical Sciences,
		National Institute of Science Education and Research,  Bhubaneswar, An OCC of Homi Bhabha National Institute,  Jatni- 752050, India}
	
	\email{bikram@niser.ac.in} \email{diptesh.saha@niser.ac.in}

	
	\begin{abstract} 
		In this article,  we consider actions of  $\Z_+^d$, $\R_+^d$ and finitely generated free groups on a von Neumann algebras  $M$ and prove a  version of  maximal ergodic inequality. Additionally, we establish non-commutative analogues of pointwise ergodic theorems for associated actions in the predual when M is finite.
	\end{abstract}

	\maketitle

	\section{Introduction}
	
	Dynamical systems studied in the context of measure preserving action have a long history starting from Birkhoff and von Neumann around 1930. Since then, the subject has been studied in numerous directions. Among them, the study of actions by multiparameter  semigroups and free groups 
	has attracted particular attention. This article is devoted to study ergodic theorem for such actions. 
	
	On the other hand, studying such ergodic theorems is also of paramount interest in the non-commutative probability theory and harmonic analysis. The groundbreaking work of Lance \cite{Lance1976} in 1976 yielded the first results in this direction. In this article, the author established a pointwise ergodic theorem for a state-preserving automorphism on a von Neumann algebra. Later, these results were extended to the case of a weak$^*$-continuous positive operator on von Neumann algebra by Kummerer in \cite{Kuemmerer1978} and to the case of one parameter semigroups of positive contractions by Conze and Dang-Ngoc in \cite{Conze1978}.

	In \cite{Yeadon1977} and   \cite{Yeadon1980},  the author considered a positive linear transformation $\alpha$ on $L^1(M, \tau)$ (here $M$ is a von Neumann algebra with f.n.s trace $\tau$) satisfying $0 \leq \alpha(X) \leq 1$ and a sub traciality condition (i.e, $\tau(\alpha(X)) \leq \tau(X)$) for all $X \in L^1(M, \tau) \cap M$ and $0 \leq X \leq 1$. Under these conditions, for all $X \in L^1(M, \tau)$ the  author proved  bilateral almost uniform convergence of the averages $A_n(X):= \frac{1}{n} \sum_{k=0}^{n-1} \alpha^k(X)$, which is the non commutative analogue of almost everywhere pointwise convergence for measure space. In the	process, he also established the maximal ergodic theorem in this context. In \cite{Junge2007}, Junge and Xu further extended the pointwise and maximal ergodic theorems for the action of $\R^d_+$ on  $L^p$ spaces for $1< p < \infty$.

	Later on, in \cite{Hong2021},  the authors considered a strongly continuous  action $ \alpha$ of G (of polynomial growth and having symmetric, compact generating set)  of  $\tau$-preserving automorphisms  on $M$, where $\tau$ is a f.n.s trace  on $M$. Then they  extended this action to $L^1(M, \tau)$
	and generalized the results of \cite{Yeadon1977} by proving the almost uniform convergence of the averages of the form $A_n(X):= \frac{1}{m(V^n)} \int_{V^n} \alpha_g(X) dm(g)$, for $ X \in L^1(M, \tau)$. 
	They obtained maximal ergodic theorem and various other individual ergodic theorems on tracial non-commutative $L^p$-spaces for the actions of locally compact  groups with polynomial growth.
	
	Very recently in \cite{Bik-Dip-neveu}, we studied pointwise convergence of ergodic averages for action of a group $G$ with polynomial growth. In this article, we considered an invariant state on a dynamical system $(M,G, \alpha)$, where $\alpha$ is sub-unital and then proved the pointwise ergodic theorem for the associated predual action on $L^1(M, \tau)$, where $\tau$ is a tracial state. 
	
	In the first part of the present article, we confine ourself to the case of multiparameter ergodic theorems, that is, the action of $\Z_+^d$ or $\R_+^d$ for an integer $d>1$. In \cite{Junge2007}, the authors also studied a maximal ergodic theorem and associated pointwise convergence of ergodic averages in noncommutative $L^p$ spaces for the action of $\Z_+^d$ or $\R_+^d$, where the action is assumed to be trace preserving.
	
	In the second part of the paper we study ergodic convergence of   spherical  averages associated to a sequence of   $w$-continuous maps $\sigma :=\{\sigma_n\}_{n=0}^\infty$  on a von Neumann algebra  $M$.  This was  inspired by the  study of the convergence of  spherical averages' associated  to a  free group action. To be precise let $\F_r$ denote the free group generated by $r$ elements and its inverses. Let $\phi$ be a homomorphism  from $\F_r$ to $Aut(M)$ and consider the spherical averages
	\begin{align*}
		\sigma_n(x) := \frac{1}{\abs{W_n}} \sum_{a \in W_n} \phi(a)(x), ~n \in \N,
	\end{align*}
	where $W_n$ is the set of elements of $\F_r$ of length $n$. Then it is interesting to determine the convergence of the sequence $S_n(x)= \frac{1}{n+1} \sum_{0}^{n} \sigma_k(x)$.

	Study of ergodic theorems for actions of groups which are not amenable were first initiated in classical setting by Arnold and Krylov in \cite{arnold1963equidistribution}. After that, it was carried out in great detail by many authors. We refer to \cite{grigorchuk1986individual}, \cite{guivarc1969generalisation}, \cite{nevo1994harmonic}, \cite{nevo1994generalization} and references therein. The free group action in the noncommutative setting was first considered by Walker in \cite{Walk:1997}. He proved that if the homomorphism $\phi$ is invariant under a faithful, normal state $\rho$, then for all $x \in M$ the sequence $S_n(x)$ will converge almost uniformly to an element  $\hat{x}$ in $M$. In fact, Walker dealt with more general maps on $M$. In particular, he considered a sequence $\{\sigma_n\}$ of completely positive maps on $M$ which preserves a faithful, normal state $\rho$ and satisfies $\sigma_1 \sigma_n= w \sigma_{n+1} + (1-w) \sigma_{n-1}$ for some $w \in (1/2,1]$ with $\sigma_0(x)=x$. Then he  proved the almost uniform convergence of $\frac{1}{n+1} \sum_{0}^{n} \sigma_k$ under some natural spectral condition on $\sigma_1$. In \cite{chilin2005few}, Walker's result was extended to $L^1(M, \tau)$, where $\tau$ is a faithful, normal, semifinite trace such that $\tau(\sigma_1(x)) \leq \tau(x)$ for all $x \in M \cap L^1(M, \tau)$ with $0 \leq x \leq 1$. 
	
	Later on, inspired by the seminal work of Junge and Xu (\cite{Junge2007}), in \cite{hu2008maximal}, Hu obtained a maximal ergodic theorem and associated individual ergodic theorem in  $L^p(M, \tau)$ corresponding to $\{\sigma_n\}$. In particular, he considered the sequence $\{\sigma_n\}$ as described above on $(M, \varphi)$, where $\varphi$ is a faithful, normal state, together with $\varphi \circ \sigma_n \leq \varphi$ and $\sigma_1 \circ \sigma_t^\varphi = \sigma_t^\varphi \circ \sigma_1$. Under these assumptions, Hu proved maximal and individual ergodic theorems associated with $\frac{1}{n+1} \sum_{0}^{n} \sigma_k$ in Haagerup non-commutative $L^p$ spaces for $1<p< \infty$. Furthermore, if $M$ is assumed to be semifinite with faithful, normal, semifinite trace $\tau$ (that is it is assumed that $\tau(\sigma_1(x)) \leq \tau(x)$ for all $x \in M \cap L^1(M, \tau)$ with $0 \leq x \leq 1$), then Hu recovered the result in \cite{chilin2005few} as mentioned above.

	In both cases,  the individual ergodic theorems  for state preserving maps on $L^1$ spaces  are not known. In this article, we address this situation and prove individual ergodic theorems  for the  state preserving  maps on $L^1$ spaces. 
	We like to point out that
	the techniques that are used in proving our  result are independent of the techniques that are used in \cite{Junge2007} or \cite{hu2008maximal}, and it reflects in the fact that trace preserving condition is not necessary to prove the  Theorem \ref{bau conv theorem-intro} and Theorem \ref{convg theorem in M_*s-free group-intro}.

	Now we highlight some of our important results. This article deals with ergodic averages corresponding to actions of $\Z_+^d$, $\R_+^d$, and finitely generated free groups. In these situations, we derive a suitable maximal ergodic inequality and use it to prove individual ergodic theorems. We first assume that $G$ is  either $\Z_+^d $ or  $\R_+^d $, where $d>1$. If $G=\Z_+^d$, we first consider $d$ many $w$-continuous commuting operators $T_1, \ldots, T_d$ on $M$ and then the associated action of $G$ is naturally defined as $ T_{( i_1, \cdots, i_d)} (\cdot) = T_1^{i_1} T_2^{i_2}\cdots T_d^{i_d} (\cdot)$ for $ ( i_1, \cdots, i_d) \in \Z_+^d$. If $G = \R_+^d$, then we consider a continuous action $T=\{T_t\}_{t \in \R_+^d}$ of $G$ on $M$. We will refer $(M, G, T)$ as a noncommutative dynamical system and we study the following averages.
	
	\begin{align*}
		M_a(\cdot):= 	
		\begin{cases}
			\frac{1}{a^d} \sum_{0 \leq i_1 <a} \cdots  \sum_{0 \leq i_d <a} T_1^{i_1} \cdots T_d^{i_d} (\cdot)&\text{ when  } G = \Z_+^d,~ a \in \N, \\\\
			\frac{1}{a^d} \int_{Q_a} T_t(\cdot) dt &  \text{  when } G = \R_+^d, ~a \in \R_+.
		\end{cases}
	\end{align*}
	
	Furthermore, if there exists a faithful normal state $\rho$ on $M$ such that $ \rho \circ T_g = \rho$, for all $g \in G$, then the quadruple $(G, M, T, \rho)$ is called kernel. Now  for the induced predual maps  we obtain the following maximal inequality.
	
	\begin{thm}\label{Maximal Inequality Real-1-intro}
		Let $G$ be either $\Z_+^d $ or  $\R_+^d $ and  $(M, G, T, \rho )$ be a kernel.		
		Let $\mu \in M_{* s}$ and $\epsilon>0$, then for any $N \in \N$ there exists a projection $e \in M$ such that $\rho(1-e) < \frac{\chi_d \norm{\mu}}{\epsilon} $ and 
		\begin{align*}
			\abs{M_a(\mu)(x)} \leq \epsilon \rho(x) \text{ for all } x \in (eMe)_+ \text{ and } 1 \leq a \leq N.
		\end{align*}
	\end{thm}
	
	We then use Theorem \ref{Maximal Inequality Real-1-intro} to deduce the following individual ergodic theorem.
	
	\begin{thm}\label{bau conv theorem-intro}
		Let $M$ be a finite von Neumann algebra with a faithful normal trace $\tau$ and  $(M, G,  T, \rho)$ be a kernel. Then for all $Y \in L^1(M, \tau)$, there exists $\overline{Y} \in L^1(M, \tau)$ such that $M_a(Y)$ converges to $\overline{Y}$ bilaterally almost uniformly.
	\end{thm}
	
	Next  we consider a generalized noncommutative dynamical system $(M, \sigma)$ where $\sigma = \{\sigma_n\}_{n=0}^\infty$ is a sequence of $w$-continuous maps on $M$ satisfying
	\begin{enumerate}\label{recurrence}
		\item[$A_1$.] $\sigma_1$ is completely positive, $\sigma_1(1) \leq 1$ and $\sigma_n$ is positive.
		\item[$A_2$.] $\sigma_1 \circ \sigma_n = w \sigma_{n+1} + (1-w) \sigma_{n-1}$ for all $n \in \N$ and $\sigma_0(x)=x$, where $1/2 < w < 1$,
	\end{enumerate}
	and consider the following averages.
	\begin{align*}
		S_n(x)= \frac{1}{n+1} \sum_{0}^{n} \sigma_k(x), ~ n\in \N.
	\end{align*}
	
	We further define a generalized kernel as a triple $(M, \sigma, \rho)$ , where $\rho$ is a faithful, normal state satisfying $\rho \circ \sigma_1 = \sigma_1  \text{ and } \rho(y \sigma_1(x)) = \rho(\sigma_1(y) x) \text {for all }   x, y\in M$. Such maps naturally arise to study the convergence of spherical averages associated to actions of finitely generated free groups.

	Then again  for the induced predual maps  we obtain the following maximal inequality and individual ergodic theorem.
	
	\begin{thm}\label{Maximal Inequality Real-free group-intro}
		Let 	 $(M , \sigma )$  be a generalized  noncommutative dynamical system and 
		$\rho$ be a f.n. state on $M$ such that $\rho \circ \sigma_1= \rho$. Further assume that $\mu \in M_{* s}$ and $\epsilon>0$. Then for any $N \in \N$ there exists a projection $e \in M$ such that $\rho(1-e) < \frac{C_w \norm{\mu}}{\epsilon} $ and 
		\begin{align*}
			\abs{S_n(\mu)(x)} \leq \epsilon \rho(x) \text{ for all } x \in (eMe)_+ \text{ and } 1 \leq n \leq N.
		\end{align*}
	\end{thm}

	\begin{thm}\label{convg theorem in M_*s-free group-intro}
		Let $M$ be a finite von Neumann algebra with f.n trace $\tau$ and 
		let  $(M, \sigma, \rho)$  be a generalized  kernel. Then for all $Y \in L^1(M, \tau)$, there exists $\overline{Y} \in L^1(M, \tau)$ such that $S_n(Y)$ converges to $\overline{Y}$ bilaterally almost uniformly.
	\end{thm}
	\noindent	
	For the definition of convergence in bilaterally almost uniformly of a sequence see Definition \ref{bilat}.

	Now we describe the layout of this article.  In \S2, we collect necessary definitions of non-commutative $L^p$ spaces, bilateral almost uniform convergence and the actions of semigroups and finitely generated free groups which will be essential for the subsequent sections. In sections \S3, \S4 and \S5, we establish mean, maximal and individual ergodic theorems, for the actions of $\Z^d_+$ and $\R^d_+$. Finally, in \S6, we study the similar theorems associated with a sequence of $w$-continuous positive contractions on a von Neumann algebra. This sequence of maps resonates with the spherical averages corresponding to a finitely generated free group.

	\section{Preliminaries}
	
	Throughout this article, $M$ will denote a von Neumann algebra acting on a separable Hilbert space $\CH$. We will write $M_s$ and $M_+$ to designate the collection of self-adjoint and positive elements of $M$, respectively. The projection lattice of $M$ will be denoted by $\CP(M)$.
	
	$M_*$ will represent the predual of $M$. It is a norm closed subspace of Banach space dual of  $M$. Similar to above, the  self-adjoint and positive elements of $M_*$ will be denoted by $M_{*s}$ and $M_{*+}$, respectively. 
	
	Let $\tau$ be a faithful, normal tracial weight on $M$.  A closed densely defined operator $X: \mathcal{D}(X)\subseteq \CH \to \CH$  is called  affiliated to $M$ if $u'X=Xu'$ for all unitary $u'$ in $M'$, where $M'$ is  the commutant of $ M$ in $\CB(\CH)$. If $X$ is affiliated to M, we will denote it by $X \eta M$. Now an operator $X \eta M$ is called $\tau$-measurable if for every $\epsilon >0$, there exists $e \in \CP(M)$ such that $e \CH \subseteq \mathcal{D}(X)$ and $\tau(1-e)< \epsilon$. The set of all $\tau$-measurable operators on $M$ is denoted by $L^0(M, \tau)$. The set of positive operators in $L^0(M, \tau)$ will be denoted by $L^0(M, \tau)_+$.
	
	The trace $\tau$ primarily defined on $M_+$ can be extended to $L^0(M, \tau)_+$ as an additive, positive homogeneous $(\text{ i.e, }\tau(\alpha A + B) = \alpha \tau(A) + \tau(B) \text{ for } A,B \in L^0(M, \tau)_+ \text{ and } \alpha>0 )$ function by the following formula.
	
	\noindent For $X \in L^0(M, \tau)_+$,
	\begin{align*}
		\tau(X) := \int_{0}^{\infty} \lambda d \tau(e_{\lambda}),
	\end{align*}
	where $X= \int_{0}^{\infty} \lambda d(e_{\lambda})$ is the spectral decomposition.
	
	\begin{defn}
		For $0< p \leq \infty$, the non-commutative tracial $L^p$-space on $(M, \tau)$ is defined by 
		\[ L^p(M,\tau):= 
		\begin{cases}
			\lbrace X \in L^0(M, \tau) : \norm{X}_p:= \tau(\abs{X}^p)^{1/p} < \infty \rbrace  & \text{ for } p \neq \infty, \\
			(M, \norm{\cdot}) & \text{ for } p= \infty
		\end{cases}
		\]
	\end{defn}
	
	With the definition above, $L^p(M,\tau)$ satisfies many properties similar to $L^p$ spaces defined on a measure space. In this article, we confine ourself to the non-commutative $L^1$ spaces. The following proposition will play important role in the sequel.
	
	\begin{prop}
		The map $\Psi: L^1(M, \tau) \to M_* $ defined by $\Psi(X)(a)= \tau(Xa)$ for $X \in L^1(M, \tau), a \in M$ is a surjective linear isometry. Furthermore,  $X$ is positive if and only if $\Psi(X)$ is positive.
	\end{prop}
	
	Now we define bilateral almost uniform convergence of a sequence of measurable operators. This type of convergence is one of the  key components of this article. Due to Egorov's theorem, this definition of convergence is actually refers to the almost everywhere convergence in classical measure space. In sections 4 and 5, we shall demonstrate this form of convergence of various sequences. 
	
	\begin{defn}\label{bilat}
		A sequence of operators $\lbrace X_n \rbrace_{n \in \N} \subseteq L^0(M, \tau)$ converges bilaterally almost uniformly (b.a.u) to $X \in L^0(M, \tau)$ if for every $\epsilon > 0$ there exists $e \in \CP(M)$ with $\tau(1-e) < \epsilon$ and
		\begin{align*}
			\lim_{n \to \infty} \norm{e (X_n - X) e} =0.
		\end{align*}
	\end{defn}
	
	We put down the following proposition for our future reference. The proof is simple, and hence we omit it.
	
	\begin{prop}\label{bau convergence closed under addition and mult}
		Suppose $\{X_n \}_{n \in \N}$ and $\{ Y_n \}_{n \in \N}$ are two sequences in $L^0(M, \tau)$ such that $\{X_n \}_{n \in \N}$ converges in measure (resp. b.a.u) to $X$ and $\{ Y_n \}_{n \in \N}$ converges in measure (resp. b.a.u) to $Y$. Then, for all  $c \in \C$,  $\{ c X_n + Y_n \}_{n \in \N}$ converges in measure (resp. b.a.u) to $cX + Y$.
	\end{prop}

	Throughout this article, unless otherwise mentioned, $G$ will denote a semigroup and $m$ be a $\sigma$-finite measure on $G$, which is both left and right invariant (i.e. $m(uB)= m(B)$ and $m(Bu)= m(B)$ for all $u \in G$).  In the sequel $\K$ will always denote either $\Z_+$ or $\R_+$. We consider a collection $\{I_l\}_{l \in \K}$ of measurable subsets of $G$ having the following properties. 
	\begin{enumerate}
		\item[(P1)] $0< m(I_l) < \infty$ for all $l \in \K$.
		\item[(P2)] $\lim_{l \to \infty} \frac{m(I_l \Delta I_l u)}{m(I_l)}=0$ and $\lim_{l \to \infty} \frac{m(I_l \Delta uI_l)}{m(I_l)}=0$ for all $u \in G$.
	\end{enumerate}
	
	Recall that a real Banach space $E$ paired with a closed convex subset $K$ satisfying $K \cap -K = \{0\}$ and $\lambda K \subseteq K$ for $\lambda \geq 0$ induces a partial order on $E$ given by $x \leq y$ if and only if $y-x \in K$ where $x,y \in E$. With such a partial order, $E$ will be referred to as an ordered Banach space. In our context, it is easy to verify that $M, M^*, M_*$, $L^p(M, \tau)$ for $1 \leq p \leq \infty$ become ordered  Banach spaces with respect to the natural order. Let us now define an action of $G$ on ordered Banach spaces.
	
	\begin{defn}\label{action on Banach sp}
		Let $E$ be an ordered Banach space. A map $\Lambda$ defined by 
		\begin{align*}
			G \ni g \xrightarrow[]{\Lambda} \Lambda_g \in \CB(E)
		\end{align*}
		is called an action if $\Lambda_g \circ \Lambda_h= \Lambda_{gh}$ for all $g,h \in G$. It is called anti-action if $\Lambda_g \circ \Lambda_h= \Lambda_{hg}$ for all $g,h \in G$. In this article, we consider both actions and anti-actions $\Lambda= \{\Lambda_g\}_{g \in G}$ which satisfy the following conditions.
		\begin{enumerate}
			\item[(C)] For all $x \in E$, the map $g \rightarrow \Lambda_g(x)$ from $G$ to $E$ is continuous. Here we take $w^*$-topology when $E=M$ and norm topology otherwise.
			\item[(UB)] $\sup_{g \in G} \norm{\Lambda_g} < \infty$.
			\item[(P)] For all $g \in G$ and $x \in E$ with $x \geq 0$, $\Lambda_g(x) \geq 0$.
		\end{enumerate}
		We refer the triple $(E, G, \Lambda)$  as a non-commutative dynamical system.  
	\end{defn}
	
	Let  us consider a non-commutative dynamical system $ (L^1(M, \tau), G, \alpha)$. Note that for each $g \in G$, $\alpha_g: L^1(M, \tau) \rightarrow L^1(M, \tau)$ is a bounded operator. 
	We recall  the dual of $\alpha_g$, denoted by $\alpha_g^*$, defined on $M= {L^1(M, \tau)}^*$ and it is determined by  the following equation;
	\begin{align}\label{dual action}
		\tau( \alpha^*_g(x)Y) = \tau( x\alpha_g(Y))  \text{ for all } x \in M  \text{ and }  Y \in L^1(M, \tau).
	\end{align}
	On the other hand if $ \beta$ is an action of $G$ on $M$, then for all $g \in G$, the predual transformation of $\beta_g$, denoted by $\hat{\beta_g}$, defined on $L^1(M, \tau)$ and it  is determined by the following equation;
	\begin{align}\label{predual action}
		\tau( \beta_g(x)Y) = \tau( x\hat{\beta}_g(Y))  \text{ for all } x \in M  \text{ and }  Y \in L^1(M, \tau).
	\end{align}
	Further, we have  $ {\hat{(\beta_g)}}^* = \beta_g$. This identification will be used in the sequel.\\

	\section{Mean convergence for semigroup action}

	The following known result provides a norm dense subset of $N_{* +}$ associated with a f.n state  $\varphi$ on $N$ and reader may find a proof in \cite[lemma 3.19]{Hiai2020}.
	
	\begin{thm}\label{dense subset of normal functionals}
		Let $N$ be a von Neumann algebra with a f.n state $\varphi$. Then $\{ \psi \in N_{* +} : \psi \leq k \varphi \text{ for some } k>0 \} $ is a norm dense subset of $N_{* +}$.
	\end{thm}
	
	Following is the Radon Nykodym type result, which will be used in the sequel  and a proof can be found in \cite{stratila79}.  
	
	\begin{lem}\label{au conv on dense set-0}
		Let $M \subseteq \CB(\CH)$ be a von Neumann algebra and $\omega$ be a positive linear functional on $M$. Suppose there exists a    $C> 0$ and  $\xi \in \CH$  such that $\omega(x) \leq C \inner{x \xi}{\xi}_{\CH}$ for all $x \in M_+$, then there exists a $x' \in M'_+$ such that $x' \leq C$ and $\omega(x)= \inner{x' x\xi}{\xi}_{\CH}$ for all $x \in M$. Moreover, if $\xi$ is a cyclic vector for $M$ in $\CH$, then $x' $ is unique.
	\end{lem}
	
	\begin{lem}\label{au conv on dense set-1}
		Let $(M, G,  T)$ be a non-commutative dynamical system and $\rho$ be a $G$-invariant f.n state on $M$. 
		Then there exists a non-commutative dynamical system $(M', G,  T')$
		such that 
		\begin{align*}
			\langle  T_g'(y')x \Omega_\rho, \Omega_\rho \rangle_{\rho}  = \langle y'  T_{g}(x) \Omega_\rho, \Omega_\rho \rangle_{\rho} \text{ for all }   x \in M, y' \in M', g \in G .
		\end{align*}
	\end{lem}
	
	\begin{proof}
		For $y' \in M_+'$ and $g \in G$, consider the linear functional $\nu_{y'}^g : M \rightarrow \C $  defined by
		\begin{align*}
			\nu_{y'}^g(x) = \langle y'  T_{g}(x) \Omega_\rho, \Omega_\rho \rangle_{\rho}.
		\end{align*}
		For $x \in M_+$, note that
		\begin{align*}
			\nu_{y'}^g(x) &= \langle y'  T_{g}(x) \Omega_\rho, \Omega_\rho \rangle_{\rho}\\
			& \leq \norm{y'}_{\infty} \langle   T_{g}(x) \Omega_\rho, \Omega_\rho \rangle_{\rho}\\
			&= \norm{y'}_{\infty} \rho( T_{g}(x)) \\
			&= \norm{y'}_{\infty} \rho(x) \\
			& = \norm{y'}_{\infty} \langle  x \Omega_\rho, \Omega_\rho \rangle_{\rho}.
		\end{align*}
		Hence, by Lemma \ref{au conv on dense set-0} there exists  a unique  $z' \in M'$ such that $\nu_{y'}^g(x) = \langle  z'x \Omega_\rho, \Omega_\rho \rangle_{\rho}$. Write $ z' =  T_g'(y')$ and then  we have $   \langle y'  T_{g}(x) \Omega_\rho, \Omega_\rho \rangle_{\rho} = \langle  T_g'(y') x \Omega_\rho, \Omega_\rho \rangle_{\rho}$. 
	\end{proof}
	
	The following is the mean ergodic type theorem for the action of  semigroup. This will be useful for  the subsequent results.  
	\begin{thm}\label{mean ergodic thm}
		Let $(M, G,  T)$ be a non-commutative dynamical system.
		Suppose there exists a  f.n state $\rho$  satisfying  $\rho( T_g(x)^2 )\leq \rho(x^2)$ for all $x \in M_s$ and $g \in G$. Then for all $\mu \in M_*$, there exists a $\bar{\mu} \in M_*$ such that  
		\begin{align*}
			\bar{\mu}= \norm{\cdot}_1- \lim_{l \to \infty} B_l(\mu),
		\end{align*}
		where for $l \in \K$, $B_l(\mu):= \frac{1}{m(I_l)} \int_{I_l} \beta_g(\mu) dm(g)$ and $\beta_g(\mu)= \mu \circ  T_g $ for all $g \in G$ and $\mu \in M_*$.
	\end{thm}
	
	\begin{proof}
		Let $L^2(M_s, \rho)$ be the closure of $M_s$ with respect to the norm induced from the inner product $ \langle \cdot, \cdot \rangle_\rho $.   
		Then we can define the following contractions on the Hilbert space $L^2(M_s, \rho)$.
		\begin{align*}
			u_g(x \Omega_{\rho})=  T_g(x) \Omega_{\rho}, x \in M_s, g \in G.
		\end{align*}
		For $l \in \K$, consider  $T_l:= \frac{1}{m(I_l)} \int_{I_l} u_g^*  \  dm(g) $. 	Then by von Neumann mean ergodic theorem,  it follows that for all $\xi \in L^2(M_s, \rho)$, $T_l(\xi)$  converges to $P \xi$ strongly, where $P$ is the orthogonal projection of $L^2(M_s, \rho)$ onto the subspace $\{ \xi \in L^2(M_s, \rho): u_g^* \xi= \xi \text{ for all } g \in G \}$.
		
		\noindent Now let $ x \in M$. Further write $x$ as $ x_1+ ix_2 $, where $x_1 , x_2\in M_s$ and 
		then by the previous argument, it follows that $ T_l(x\Omega):= T_l(x_1\Omega_{\rho}) + i T_l(x_2\Omega_{\rho})$ converges in $L^2(M, \rho)$.
		
		\noindent Let $y_1, y_2 \in M'(\sigma)$ and define $\psi_{y_1, y_2}(x)= \inner{x y_1 \Omega_{\rho}}{ y_2 \Omega_{\rho}}$ for all $x \in M$. Then there exists a  $z \in M$ such that $y_1^*y_2 \Omega_{\rho}= z \Omega_{\rho}$, where $z= J \sigma_{i/2}(y_2^*y_1)J$. Consequently for $x \in M$,
		\begin{align*}
			B_l(\psi_{y_1, y_2})(x)
			& = \frac{1}{m(I_l)} \int_{I_l} \inner{ T_{g}(x) y_1 \Omega_{\rho}}{y_2 \Omega_{\rho}}_{\rho} dm(g) \\
			& = \frac{1}{m(I_l)} \int_{I_l} \inner{ T_{g}(x) \Omega_{\rho}}{y_1^* y_2 \Omega_{\rho}}_{\rho} dm(g) \\
			& = \frac{1}{m(I_l)} \int_{I_l} \inner{ T_{g}(x) \Omega_{\rho}}{z\Omega_{\rho}}_{\rho} dm(g)\\
			& = \frac{1}{m(I_l)} \int_{I_l} \inner{x \Omega_{\rho}}{ u_{g}^* (z\Omega_{\rho})}_{\rho} dm(g)\\
			& = \inner{x \Omega_{\rho}}{ T_l(z\Omega_{\rho})}_{\rho}.
		\end{align*}
		Hence, for all $x \in M$, $B_l(\psi_{y_1, y_2})(x) \rightarrow \inner{x \Omega_{\rho}}{\eta}_{\rho}$, where $ \eta =\lim_{l \to \infty} T_l(z\Omega_{\rho})$. Consider $\overline{\psi}_{y_1, y_2} \in M_*$,  defined by 
		\begin{align*}
			\overline{\psi}_{y_1, y_2}(x)= \inner{x \Omega_{\rho}}{\eta }_{\rho}, ~~x \in M.
		\end{align*}
		Then by  standard argument it follows that $\overline{\psi}_{y_1, y_2} \circ  T_g = \overline{\psi}_{y_1, y_2}$ for all $g \in G$ and 
		\begin{align*}
			\overline{\psi}_{y_1, y_2} = \norm{\cdot}_1- \lim_{l \to \infty} B_l(\psi_{y_1, y_2}).
		\end{align*}
		Hence the result follows since the set $\{ \psi_{y_1, y_2} : y_1, y_2 \in M'(\sigma) \}$ is total in $M_*$.
	\end{proof}

	\section{Maximal Inequality for the actions of $\Z_+^d$ and $\R_+^d$ }
	In this section we consider state preserving actions of  $\Z_+^d$ and $\R_+^d$ on a von Neumann algebra $M$. Then prove  a version of  maximal inequality for the induced action on the predual von Neumann algebra $M_*$. We start with the following useful lemma which is well known in the literature. For a proof we refer to \cite[Lemma 4.1]{Bik-Dip-neveu}.
	
	\begin{lem}\label{imp lem for maximal inequality}
		Let $M\subseteq \CB(\CH)$ be a von Neumann algebra and $c \in M$ such that $0 \leq c \leq 1$. Suppose $x \in M$ such that  $0 \leq x \leq s(c)$, then 
		$$x \in \overline{\{ y \in M : 0 \leq y \leq n c \text{ for some } n \in \N \}}^{w\text{-topology}}. $$ 
	\end{lem}
	
	For a semigroup action on a von Neumann algebra,  to prove  a  maximal inequality for the  induced  action of semigroups on the predual, we need an auxiliary  maximal inequality type theorem for the single induced map $T^*$ on $M_*$, where $T$ is a positive, sub-unital  and state preserving map on $M$.
	Such maximal inequality is already   established in  \cite{Bik-Dip-neveu},  following  similar ideas as  in \cite{Yeadon1977}, but in a  different context. We include the statement and for  
	the completeness of this article, we write  a proof. 
	
	\begin{thm}\label{maximal lemma}
		Let $M$ be a von Neumann algebra with a f.n state $\rho$  on $M$. Assume that $T: M \to M$ is a sub-unital, $\rho$-preserving and $w$-continuous positive linear map. Let $\mu \in M_{*s}$ and $\epsilon>0$. Then for any $N \in \N$ there exists  $e \in \CP(M)$ such that $\rho(1-e) < \norm{\mu}/\epsilon $ and
		\begin{align*}
			\abs{S_n(\mu)(x)} \leq \epsilon \rho(x) \text{ for all } x \in (eMe)_+ \text{ and } n \in \{1,\ldots, N \}.
		\end{align*}
	\end{thm}
	
	\begin{proof}
		Let $N \in \N$ and consider the $w$-compact subset $\mathcal{L}$ of the von Neumann algebra $R:=\bigoplus_{n=1}^{2N} M$ defined by
		\begin{align*}
			\mathcal{L}:= \Big\{ (\underbar{x}, \underbar{y}):= (x_1,\ldots, x_N,y_1,\ldots, y_N) \in R : x_i, y_i \geq 0 \  \text{ and } \sum_{i=1}^{N} (x_i + y_i) \leq 1 \Big\}.
		\end{align*}
		Now consider the $w$-continuous linear functional $\kappa$ on $R$ defined by
		\begin{align*}
			\kappa((\underbar{x}, \underbar{y})):= \sum_{n=1}^{N} n \Big[ S_n(\mu')(x_n) - \rho(x_n) - S_n(\mu')(y_n) - \rho(y_n) \Big], \text{ for all } (\underbar{x}, \underbar{y}) \in R,
		\end{align*}
		
		where, $\mu':= \mu/ \epsilon$. Since $\mathcal{L}$ is $w$-compact,  the supremum  value of the function $\kappa$ on the set $\mathcal{L}$ will be attained. Let $(\underbar{x}, \underbar{y}) \in \mathcal{L}$ be such that $\kappa((\underbar{x}, \underbar{y})) \geq \kappa ((\underbar{a}, \underbar{b}))$ for all $(\underbar{a}, \underbar{b}) \in \mathcal{L}$.
		We define $c:= 1- \sum_{n=1}^{N} (x_i + y_i)$. Note that $c \geq 0$. Let $c'$ be an element in $M$ such that $0 \leq c' \leq c$. For $m \in \{ 1, \ldots ,N \}$, take $\underline{x}' = (x_1, \ldots, x_m+c', \ldots, x_N) $. We consider the  point $(\underbar{x}', \underbar{y}) \in \mathcal{L}$, i.e, 
		\begin{align*}
			(\underbar{x}', \underbar{y})= (x_1, \ldots, x_m+c', \ldots, x_N,  y_1, \ldots, y_N).
		\end{align*}
		Then we have,
		\begin{align*}
			\kappa ((\underbar{x}, \underbar{y})) \geq \kappa ((\underbar{x}', \underbar{y})).
		\end{align*}
		Therefore, we deduce the following
		\begin{align}\label{a}
			\nonumber&  \kappa ((\underbar{x}, \underbar{y})) - \kappa ((\underbar{x}', \underbar{y})) \geq 0 \\
			\nonumber \Rightarrow \qquad &\sum_{n=1}^{N} n \Big[ S_n(\mu')(x_n) - \rho(x_n) - S_n(\mu')(y_n) - \rho(y_n) \Big] \\
			\nonumber   - &\sum_{n=1}^{N} n \Big[ S_n(\mu')(x_n) - \rho(x_n) - S_n(\mu')(y_n) 
			- \rho(y_n) \Big]- m \Big[ S_m(\mu')(c') - \rho(c') \Big]\geq 0 \\
			\nonumber \Rightarrow \qquad&  - m \Big[ S_m(\mu')(c') - \rho(c') \Big]   \geq 0  \\
			\Rightarrow \qquad&  S_m(\mu')(c') \leq \rho(c').
		\end{align}
		We again consider the following point in $ \mathcal{L}$ 
		\begin{align*}
			(\underbar{x}, \underbar{y}'):= (x_1, \ldots,x_N, y_1, \ldots, y_m+c', \dots, y_N).
		\end{align*}
		Then we have
		\begin{align*}
			\kappa ((\underbar{x}, \underbar{y})) \geq \kappa ((\underbar{x}, \underbar{y}')).
		\end{align*}
		Likewise as  eq. \ref{a}, we deduce 
		\begin{align}\label{a'}
			-S_m(\mu')(c') \leq  \rho(c') .
		\end{align}
		Hence by eq. \ref{a} and eq. \ref{a'}, we have
		\begin{align}\label{b}
			\abs{S_m(\mu)(c')} \leq \epsilon \rho(c') \text{ for all } c' \in M \text{ such that } 0 \leq c' \leq c .
		\end{align}
		Now if $x \in M$ with $0 \leq x \leq 1$ and $x \leq n c$ for some $n \in \N$, consider $c':= \frac{x}{n}$. As, $ c'\leq c$, hence by eq. \ref{b} we have
		\begin{align*}
			\abs{S_m(\mu)(x)} \leq \epsilon \rho(x).
		\end{align*}
		
		Let $e=s(c)$ and $x \in eMe$ with $0 \leq x \leq e$. Since $\abs{S_m(\mu)(\cdot)},~ \rho$ are w-continuous, so by Lemma \ref{imp lem for maximal inequality}, we have
		\begin{align*}
			\abs{S_m(\mu)(x)} \leq \epsilon \rho(x).
		\end{align*}
		And consequently we achieve
		\begin{align*}
			\abs{S_n(\mu)(x)} \leq \epsilon \rho(x) \text{ for all } x \in (eMe)_+ \text{ and } n \in \{1,\ldots, N \}.
		\end{align*}
		We claim that $\rho(1-e) < \norm{\mu}/\epsilon $. 
		Consider  the point $(Tx_2, \ldots, Tx_N, 0,Ty_2, \ldots, Ty_N,0) \in R$. Since $T(1) \leq 1$, we note that  $(Tx_2, \ldots, Tx_N, 0,Ty_2, \ldots, Ty_N,0)  \in \mathcal{L}$.
		Hence,  we have  
		\begin{align*}
			\kappa((\underbar{x}, \underbar{y})) \geq \kappa((Tx_2, \ldots, Tx_N, 0,Ty_2, \ldots, Ty_N,0)  ).
		\end{align*}
		From this we deduce the following
		\begin{align*}
			& \kappa((\underbar{x}, \underbar{y})) \geq \kappa((Tx_2, \ldots, Tx_N, 0,Ty_2, \ldots, Ty_N,0)  ) \\
			\Rightarrow 	\qquad&  \mu'(x_1)- \rho(x_1) + \cdots + \sum_{k=0}^{N-1}\mu'(T^k(x_N))- N\rho(x_N)  \\
			-	&\mu'(y_1)- \rho(y_1) - \cdots - \sum_{k=0}^{N-1}\mu'(T^k(y_N)) - N\rho(y_N) \\
			&\geq \\
			&  \mu'(Tx_2)- \rho(Tx_2) + \cdots + \sum_{k=0}^{N-2}\mu'(T^{k+1}x_N)- (N-1)\rho(Tx_N)  \\
			-&\mu'(Ty_2)- \rho(Ty_2) -\cdots - \sum_{k=0}^{N-2}\mu'(T^{k+1}y_N)- (N-1)\rho(Ty_N) \\
			\Rightarrow \qquad  	&\mu'(x_1)+ \sum_{k=0}^{1}\mu'(T^k(x_2)) +\cdots + \sum_{k=0}^{N-1}\mu'(T^k(x_N))  - \sum_{k=1}^{N} k\rho(x_k) \\
			- 	&\mu'(Tx_2)   -  \sum_{k=0}^{1}\mu'(T^{k+1}x_k)-\cdots  -  \sum_{k=0}^{N-2}\mu'(T^{k+1}x_N)  +\sum_{k=1}^{N} (k-1)\rho(x_k)  \\
			& \geq \\
			&  \mu'(y_1) + \sum_{k=0}^{1}\mu'(T^k(y_2))+ \cdots +\sum_{k=0}^{N-1}\mu'(T^k(y_N))+  \sum_{k=1}^{k}k\rho(y_k) \\
			-& \mu'(Ty_2)- \sum_{k=0}^{1}\mu'(T^{k+1}y_3) -\cdots - \sum_{k=0}^{N-2}\mu'(T^{k+1}y_N)- \sum_{k=1}^{N} (k-1)\rho(y_k) \\
			\Rightarrow \qquad 	& \mu'(\sum_{k=1}^{N}x_k) -\rho(\sum_{k=1}^{N}x_k)  \geq \mu'( \sum_{k=1}^{N}y_k ) + \rho(\sum_{k=1}^{N}y_k) \\
			\Rightarrow \qquad 	& \mu'(\sum_{k=1}^{N}x_k - \sum_{k=1}^{N}y_k ) \geq \rho(\sum_{k=1}^{N}x_k + \sum_{k=1}^{N}y_k) \\
			\Rightarrow \qquad 	&\norm{\mu}/ \epsilon \geq \rho(\sum_{i=1}^N (x_i + y_i)) \\
			\Rightarrow \qquad 	& \rho(c) \geq 1- \norm{\mu}/ \epsilon, \quad (\text{since } c= 1- \sum_{i=1}^N (x_i + y_i))\\
			\Rightarrow \qquad 	&\rho(e) \geq 1- \norm{\mu}/ \epsilon, \quad (\text{since } e=s(c)).
		\end{align*}
		This completes the proof.
	\end{proof}
	
	\begin{defn}
		Let $(M, G,  T)$ be a non-commutative dynamical system and $\rho$ be  a f.n. state on $M$. Then  $(M, G,  T, \rho)$ is called kernel if 	
		\begin{enumerate}
			\item $\rho$ is $G$-invariant and 
			\item  $ T$ is sub-unital, i.e, $ T_g(1)\leq 1$ for all $g \in G$. 
		\end{enumerate}
	\end{defn}
	
	\begin{rem}\label{rem-kadison}
		Let  $(M, G,  T, \rho)$  be a kernel, then observe the following. 
		\begin{enumerate}
			\item  By Kadison's inequality \cite{Kadison:1952vc}, for all $x \in M_s$ we have $ T_g(x)^2 \leq  T_g(x^2)$, and further as $\rho$ is $G$-invariant, it follows that $$\rho( T_g(x)^2 )\leq \rho(x^2) \text{ for all } x \in M_s \text{ and } g \in G.$$
			\item As $ T_g$ is a positive map on $M$, so by Russo-Dye theorem \cite[Corollary 2.9]{Paulsen:2002tj} $\norm{ T_g} =\norm{ T_g(1)} \leq 1$. Thus, we have   $\text{sup}_{g\in G} \norm{ T_g} \leq  1$.	  
		\end{enumerate}	
	\end{rem}
	
	\subsection{Action of $\Z^d_+$}
	In this subsection, we consider $(M, \Z^d_+, T)$, the non-commutative dynamical system asociated to $\Z^d_+ $ actions. Then note that there exists positive $d$-commuting maps  $ T_1, T_2, \cdots ,T_d$ on $M$ such that $ T_{( i_1, \cdots, i_d)} (\cdot) = T_1^{i_1} T_2^{i_2}\cdots T_d^{i_d} (\cdot)$ for $ ( i_1, \cdots, i_d) \in \Z_+^d$. Then we prove a maximal inequality for the induced action on $M_*$. 
	An analogue of the following result is proved in \cite[Theorem 3.4, pp-213]{Krengel1985} for the case of classical $L^1$ spaces. In our setup we require the exact same result in a general ordered Banach space $E$. Although the proof is similar, we write it for the sake of completeness.
	We recall the following  known result which is the main ingredient in this context.  	
	\begin{lem}\cite{Krengel1985}\label{Krengel1985}
		Let $\xi(x)= 1 - \sqrt{1-x}$ for $0 \leq x \leq 1$. For $n \in \N$, write $[\xi(x)]^n = \sum_{p=0}^{\infty} \alpha_p^{(n)} x^p$. Then
		\begin{enumerate}
			\item[(i)] $\alpha_p^{(n)} =
			\begin{cases}
				0 &\text{ for } p<n \\
				\frac{n}{2p} 2^{n+1-2p} {2p-n-1 \choose p-1} &\text{ for } p\geq n .
			\end{cases}$
			\item[(ii)] If $\varphi(n)$ denotes the greatest integer less than or equal to $\sqrt{n} + 1$, then there exists $c>0$ such that 
			\begin{align*}
				\frac{1}{\varphi(n)} \sum_{0 \leq j < \varphi(n)} \alpha_{v +j}^{(j)} \alpha_{w +j}^{(j)} \geq \frac{c}{n^2} \text{ holds for all } 0 \leq v,w <n.
			\end{align*}
		\end{enumerate}
	\end{lem}
	First we fix the following notation. For for an integer $d>1$ notice that there is an unique $m \in \N$ such that $2^{m-1}< d \leq 2^m$. For $n \in \N$ and $d>1$ we fix $n_d := \varphi^m(n)$, where $\varphi^m:= \varphi \circ \cdots \circ \varphi$($m$ times).

	\begin{thm}\label{erg inq 1}
		Let $E$ be an ordered Banach space and $d>1$ be an integer. Then there exists $\chi_d>0$ and a family $\{a(u): u=(u_1, \ldots, u_d) \in \Z_+^d\}$ of strictly positive numbers summing to $1$ such that the following holds: If $T_1, \ldots, T_d$ are commuting positive contractions of $E$, then the operator 
		\begin{align*}
			U= \sum_{u \in Z_+^d} a(u) T_1^{u_1} \cdots T_d^{u_d}
		\end{align*}
		satisfies
		\begin{align*}
			\frac{1}{n^d} \sum_{0 \leq i_1 <n} \cdots  \sum_{0 \leq i_d <n} T_1^{i_1} \cdots T_d^{i_d} f \leq \frac{\chi_d}{n_d} \sum_{j=0}^{n_d -1} U^j f 
		\end{align*}
		for all $n \in \N$ and $f \in E_+$.
	\end{thm}
	
	\begin{proof}
		Enough to consider $d=2^m$ for one can put $T_j=I$ for all $d < j \leq 2^m$ if $d< 2^m$.
		Let us denote by $\overline{\xi}(x)= \frac{\xi(x)}{x}$ for $0 < x \leq 1$. Then, by (i) of  Lemma \ref{Krengel1985} we have 
		\begin{align*}
			\overline{\xi}(x) = \sum_{\lambda =0}^{\infty} \alpha_{\lambda +1}^{(1)} x^\lambda.
		\end{align*}
		Now if $T$ is any contraction, then $\overline{\xi}(T)= \sum_{\lambda =0}^{\infty} \alpha_{\lambda +1}^{(1)} T^\lambda$ is again a contraction. Consequently, since $\alpha_\lambda^{(j)}=0$ for all $\lambda < j$ we have
		\begin{align*}
			(\overline{\xi}(T))^j = \sum_{\lambda =0}^{\infty} \alpha_{j + \lambda}^{(j)} T^\lambda
		\end{align*}
		for every $j \in \N$. When $d=2$ we first consider the contraction 
		\begin{align*}
			U= \overline{\xi}(T_1) \circ \overline{\xi}(T_2).
		\end{align*}
		Note that since $T_1$ and $T_2$ commutes we have
		\begin{align*}
			\frac{1}{\varphi(n)} \sum_{j < \varphi(n)} U^j 
			&\geq \frac{1}{\varphi(n)}  \sum_{0 \leq i_1, i_2< n} \sum_{j < \varphi(n)} \alpha_{j + i_1}^{(j)} \alpha_{j + i_1}^{(j)} T_1^{i_1} T_2^{i_2} \\
			& \geq \frac{c}{n^2} \sum_{0 \leq i_1, i_2< n} T_1^{i_1} T_2^{i_2} ~~(\text{by Lemma } \ref{Krengel1985}~(ii)).
		\end{align*}
		Therefore, the result holds for $d=2$ with $\chi_d= 1/c$. Next, assume that $d=4$. In this case consider the contraction 
		\begin{align*}
			U= \overline{\xi}(U_1) \circ \overline{\xi}(U_2),
		\end{align*}
		where $U_1= \overline{\xi}(T_1) \circ \overline{\xi}(T_2)$ and $U_2= \overline{\xi}(T_3) \circ \overline{\xi}(T_4)$. Hence, we obtain
		\begin{align*}
			\frac{1}{\varphi(\varphi(n))} \sum_{0 \leq j < n_4} U^j 
			&\geq \frac{1}{\varphi(\varphi(n))} \sum_{0 \leq j < \varphi(n)} U^j \\
			&\geq \frac{1}{\varphi(\varphi(n))} \sum_{0 \leq i_1, i_2< \varphi(n)} \sum_{j < \varphi(\varphi(n))} \alpha_{j + i_1}^{(j)} \alpha_{j + i_1}^{(j)} U_1^{i_1} U_2^{i_2}\\
			& \geq \frac{c}{\varphi(n)^2} \sum_{0 \leq i_1, i_2< \varphi(n)} U_1^{i_1} U_2^{i_2} ~~(\text{by Lemma } \ref{Krengel1985}~(ii)) \\
			& \geq c\Big( \frac{1}{\varphi(n)} \sum_{0 \leq i_1< \varphi(n)} U_1^{i_1} \Big) \circ \Big( \frac{1}{\varphi(n)} \sum_{0 \leq i_2< \varphi(n)} U_2^{i_2} \Big) \\
			& \geq \frac{c^3}{n^4} \sum_{0 \leq i_1, i_2, i_3, i_4< n} T_1^{i_1} T_2^{i_2} T_3^{i_3} T_4^{i_4}.
		\end{align*}
		Therefore, the result holds for $d=3,4$ with $\chi_3= \frac{1}{c^3}$ and $\chi_4= \frac{1}{c^4}$ respectively.
		
		Now if for $d>1$ the contractions $\{T_1, \ldots, T_{2d}\}$ are given, then following the method described above, first construct a contraction $U_1$ using $\{T_1, \ldots, T_d\}$, and then construct another contraction $U_2$ using $\{T_{d+1}, \ldots, T_{2d}\}$. Finally consider $U:= \overline{\xi}(U_1) \circ \overline{\xi}(U_2)$.
		
		Moreover, observe that it follows from the construction that 
		$$U= \sum_{u \in Z_+^d} a(u) T_1^{u_1} \cdots T_d^{u_d},$$
		where $a(u)>0$ for all $u \in \Z_+^d$ and $\sum_{u \in Z_+^d} a(u)=1$.
	\end{proof}
	
	\begin{defn}
		For $n \in \Z_+$,
		$$A_n(f) := \frac{1}{n^d} \sum_{0 \leq i_1 <n} \cdots  \sum_{0 \leq i_d <n} T_1^{i_1} \cdots T_d^{i_d} (f),~~ f \in E. $$
	\end{defn}

	\begin{thm}
		Let $(M, \Z_d^+, T, \rho )$ be a kernel. Also assume $\mu \in M_{* s}$ and $\epsilon>0$, then for any $N \in \N$ there exists a projection $e \in M$ such that $\rho(1-e) < \frac{\chi_d \norm{\mu}}{\epsilon} $ and 
		\begin{align*}
			\abs{A_n(\mu)(x)} \leq \epsilon \rho(x) \text{ for all } x \in (eMe)_+ \text{ and } n \in \{1,\ldots, N \}.
		\end{align*}
	\end{thm}
	
	\begin{proof}
		Let $\epsilon>0$ and $N \in \N$. It is enough to consider $\mu \in M_{* +}$. Then there exists $e \in \CP_0(M)$ such that $\rho(1-e) < \frac{\chi_d \norm{\mu}}{\epsilon}$ and 
		\begin{align*}
			\abs{\frac{1}{n} \sum_{j=0}^{n -1} U^j (\mu)(x)} \leq \frac{\epsilon}{\chi_d} \rho(x) \text{ for all } x \in (eMe)_+ \text{ and } n \in \{1,\ldots, N_d \}.
		\end{align*}
		Now consider $n \in \{1,\ldots,N\}$. Note that $n_d \leq N_d$. Therefore,
		\begin{align*}
			A_n(\mu)(x) \leq \frac{\chi_d}{n_d} \sum_{j=0}^{n_d -1} U^j(\mu)(x) \leq \epsilon  \rho(x) \text{ for all } x \in (eMe)_+.
		\end{align*}
	\end{proof}

	\subsection{Action of $\R^d_+$}
	In this subsection, we consider $(M, \R^d_+, T)$ and we deduce a maximal inequality for the induced action on the predual $M_*$.  We follow the same technique as it is obtained in \cite{brunel73} for the case of classical $L^1$- spaces to establish our case for the interest to ameliorate the exposition  of the  exposition of this article. 
	\begin{defn}
		Let $a \in \R_+$ and define the set $Q_a:= \{(t_1, \ldots, t_d) \in \R^d_+ : t_1<a, \ldots, t_d< a\}$. Then consider the following averages.
		
		\begin{align*}
			M_a(f):= \frac{1}{a^d} \int_{Q_a} T_t(f) dt.
		\end{align*}
	\end{defn}
	
	\begin{rem}
		In particular, when $d>1$ be an integer and $n \in \Z_+$, note that
		\begin{align*}
			M_n(f):= \frac{1}{n^d} \sum_{(j_1,\ldots, j_d) \in \Z^d_+ \cap Q_n} \int_{Q_{j_1,\ldots, j_d}} T_t(f) dt,
		\end{align*}
		where for any $(j_1,\ldots,j_d) \in \Z^d_+$, $Q_{j_1,\ldots, j_d}:= \{(t_1, \ldots, t_d) \in \R^d_+ : j_1\leq t_1<j_1+1, \ldots, j_d \leq t_d< j_d+ 1\}$. It also follows that
		\begin{align*}
			M_1(f)= \int_{Q_{0,\ldots, 0}} T_t(f) dt= \int_{[0,1)^d} T_t(f) dt.
		\end{align*}
		
		Further, denoting $S_1= T_{(1,0,\ldots,0,0)},~~ S_2= T_{(0,1,0,\ldots,0)}, ~~\cdots, S_d= T_{(0,0,\ldots,0,1)}$ we observe that 
		\begin{align*}
			M_n(f):= \frac{1}{n^d} \sum_{(j_1,\ldots, j_d) \in \Z^d_+ \cap Q_n} S_1^{j_1} \cdots S_d^{j_d} (M_1(f))
		\end{align*}
	\end{rem}
	
	\begin{thm}\label{erg ieq 2}
		Let $E$ be an ordered Banach space and $d>1$ be an integer. Then for all $a>0$ there exists $n \in \N$ such that 
		\begin{align*}
			M_a(f) \leq \frac{1}{n^d} \sum_{(j_1,\ldots, j_d) \in \Z^d_+ \cap Q_{n+1}}  S_1^{j_1} \cdots S_d^{j_d} M_1(f)= \frac{1}{n^d} \sum_{0 \leq j_1 <n+1} \cdots  \sum_{0 \leq j_d <n+1} S_1^{j_1} \cdots S_d^{j_d} M_1(f)
		\end{align*}
		for all $f \in E_+$.
	\end{thm}
	
	\begin{proof}
		Let $a>0$. Choose $n= [a]$ . Then 
		\begin{align*}
			M_a(f) 
			&\leq \frac{1}{n^d} \int_{Q_{n+1}} T_t(f) dt \\
			&= \frac{1}{n^d} \sum_{(j_1,\ldots, j_d) \in \Z^d_+ \cap Q_{n+1}} \int_{Q_{j_1,\ldots, j_d}} T_t(f) dt \\
			&= \frac{1}{n^d} \sum_{(j_1,\ldots, j_d) \in \Z^d_+ \cap Q_{n+1}} S_1^{j_1} \cdots S_d^{j_d} (M_1(f))\\
			&=\frac{1}{n^d} \sum_{0 \leq j_1 <n+1} \cdots  \sum_{0 \leq j_d <n+1} S_1^{j_1} \cdots S_d^{j_d} M_1(f).
		\end{align*}
	\end{proof}
	
	\begin{prop}
		Let $E$ be an ordered Banach space and $d>1$ be an integer. Then there exists $\chi_d>0$  and a family $\{a(u): u=(u_1, \ldots, u_d) \in \Z_+^d\}$ of strictly positive numbers summing to $1$ such that the following holds:
		\begin{align*}
			U= \sum_{u \in Z_+^d} a(u) S_1^{u_1} \cdots S_d^{u_d}
		\end{align*}
		and for all $a>0$ and $f \in E_+$
		\begin{align*}
			M_a(f) \leq \frac{\chi_d}{([a]+1)_d} \sum_{j=0}^{([a]+1)_d-1} U^j M_1(f) .
		\end{align*}
	\end{prop}
	
	\begin{proof}
		It is clear from Theorem \ref{erg inq 1} and Theorem \ref{erg ieq 2}.
	\end{proof}

	\begin{thm}\label{Maximal Inequality Real}
		Let $(M, \R^d_+, T, \rho )$ be a kernel. Also assume $\mu \in M_{* s}$ and $\epsilon>0$. Then for any $N \in \N$ there exists a projection $e \in M$ such that $\rho(1-e) < \frac{\chi_d \norm{\mu}}{\epsilon} $ and 
		\begin{align*}
			\abs{M_a(\mu)(x)} \leq \epsilon \rho(x) \text{ for all } x \in (eMe)_+ \text{ and } 1 \leq a \leq N.
		\end{align*}
	\end{thm}
	
	\begin{proof}
		Let $\epsilon>0$ and $N \in \N$. Enough to consider $\mu \in M_{* +}$. Then there exists $e \in \CP_0(M)$ such that $\rho(1-e) < \frac{\chi_d \norm{\mu}}{\epsilon}$ and 
		\begin{align*}
			\abs{\frac{1}{n} \sum_{j=0}^{n -1} U^j (M_1(\mu))(x)} \leq \frac{\epsilon}{\chi_d} \rho(x) \text{ for all } x \in (eMe)_+ \text{ and } n \in \{1,\ldots, (N+1)_d \}.
		\end{align*}
		Now consider $1 \leq a \leq N$. Note that $[a]_d \leq N_d$. Therefore,
		\begin{align*}
			M_a(\mu)(x) \leq \frac{\chi_d}{([a]+1)_d} \sum_{j=0}^{([a]+1)_d -1} U^j(M_1(\mu))(x) \leq \epsilon  \rho(x) \text{ for all } x \in (eMe)_+.
		\end{align*}
	\end{proof}
	
	We assume that  $G$ is  either $\Z_+^d $ or  $\R_+^d $. Let $(G, T, M ) $ be a non commutative dynamical system and we summarize the notation as follows. 
	\begin{align*}
		M_a(\cdot):= 	
		\begin{cases}
			\frac{1}{a^d} \sum_{0 \leq i_1 <a} \cdots  \sum_{0 \leq i_d <a} T_1^{i_1} \cdots T_d^{i_d} (\cdot)&\text{ when  } G = \Z_+^d,~ a \in \N, \\\\
			\frac{1}{a^d} \int_{Q_a} T_t(\cdot) dt &  \text{  when } G = \R_+^d, ~a \in \R_+.
		\end{cases}
	\end{align*}
	
	We also combine   the obtained maximal inequality as follows; 
	\begin{thm}\label{Maximal Inequality Real-1}
		
		Let $G$ be either $\Z_+^d $ or  $\R_+^d $ and  $(M, G, T, \rho )$ be a kernel.		
		Let $\mu \in M_{* s}$ and $\epsilon>0$, then for any $N \in \N$ there exists a projection $e \in M$ such that $\rho(1-e) < \frac{\chi_d \norm{\mu}}{\epsilon} $ and 
		\begin{align*}
			\abs{M_a(\mu)(x)} \leq \epsilon \rho(x) \text{ for all } x \in (eMe)_+ \text{ and } 1 \leq a \leq N.
		\end{align*}
	\end{thm}

	\section{Pointwise convergence for semigroup action}
	
	In this section, we  assume  $G$ is  either $\Z_+^d $ or  $\R_+^d $ and  $(G, T, M ) $ be a non commutative dynamical system. Then we prove a non-commutative version of pointwise ergodic theorem for the induced action on the pre-dual $M_*$.

	%
	
	Let  $(M, G,  T, \rho)$  be a kernel. Then by Remark  \ref{rem-kadison} and Theorem  \ref{mean ergodic thm}, it follows that for all $\mu \in M_*$,  $ \lim_{n \to \infty} B_n(\mu)$ exists in $\norm{\cdot}_1$ in $M_*$. We denote the limit by $\bar{\mu}$.

	\begin{lem}\label{au conv on dense set 2}
		Let $(M, G,  T, \rho)$ be a kernel.	Consider the following set 
		\begin{align*}
			\CW_1 := \{ \nu - B_k(\nu ) + \bar{\nu} : \quad  k \in \K,\  \nu \in M_{*+} \text{ with } \nu \leq \lambda \rho \text{ for some } \lambda >0 \}.
		\end{align*}
		\begin{enumerate} 
			\item[(i)] Write $\CW= \CW_1-\CW_1$, then $\CW$ is dense in $M_{* s}$ and 
			\item[(ii)] for all $\nu \in \CW$, we have  
			\begin{align}\label{maximal type1}
				\lim_{l \to \infty} \sup_{x \in M_+, x \neq 0} \abs{(B_l(\nu)- \bar{\nu})(x)}/ \rho(x) =0.
			\end{align}
		\end{enumerate}
	\end{lem}
	
	\begin{proof}
		\emph{ (i):}	Let $\mu \in M_{* +}$ and $\epsilon>0$. From Theorem \ref{dense subset of normal functionals}, find a $\nu \in M_{*+}$ with $\nu \leq \lambda \rho $ for some $\lambda >0$, such that  $ \norm{\mu - \nu}<\epsilon/2 $. Further, by Theorem \ref{mean ergodic thm}  we know that $B_l(\nu)$ is convergent, and write  $\bar{\nu} = \lim_{l \to \infty} B_l(\nu)$. So there exists a $l_0 \in \K $ such that $\norm{ \bar{\nu} -B_{l_0}(\nu)} \leq \epsilon/2 $. Therefore by triangle inequality, we have 
		$ \norm{\mu-(\nu -B_{l_0}(\nu) + \bar{\nu})} \leq \norm{\mu- \nu} +  \norm{B_{l_0}(\nu) - \bar{\nu}} < \epsilon$.\\
		Now for $\mu \in M_{* s}$, we write $\mu= \mu_+ - \mu_-$, where $\mu_+, \mu_-$ are normal positive linear functional. Thus, it follows that $\CW$ is dense  in $ M_{* s}$.\\
		
		\noindent \emph{ (ii):} Fix $k \in \K$ and consider $ \nu_k := \nu - B_k(\nu ) + \bar{\nu} $ and it is enough to prove eq. \ref{maximal type1} for $\nu_k $. First we claim that $\bar{\nu}_k = \bar{\nu}$. 
		
		Since $\nu \leq \lambda \rho $, so by  Lemma \ref{au conv on dense set-0} there exists a unique $y_1' \in M_+'$ with $  y_1' \leq \lambda $ such that 
		\begin{align*}
			\nu(x)= \langle y_1'x \Omega_{\rho}, \Omega_{\rho} \rangle_{\rho} \text{ for all } x \in M.
		\end{align*}
		Let  $y' \in M'$, write $ B_l'(y') := \frac{1}{m(I_l)}\int_{I_l}  T_{g}'(y') dm(g)$ and by Lemma \ref{au conv on dense set-1}, we have
		\begin{align*}
			\langle B_l'(y') x \Omega_{\rho}, \Omega_{\rho} \rangle_{\rho} = \langle y' B_l(x) \Omega_{\rho}, \Omega_{\rho} \rangle_{\rho}, \quad x \in M,  y' \in M'.
		\end{align*}	
		Now  for all $l \in \K$ and $x \in M_+$, we note that 
		\begin{align*}
			\abs{(B_l(\nu_k) - \bar{\nu})(x)}
			&= \abs{(\nu_k - \bar{\nu})(B_l(x))}\\
			&= \abs{(\nu - B_k(\nu))(B_l(x))} ~~ \text{ (since } \nu_k := \nu - B_k(\nu ) + \bar{\nu})\\
			&= \abs{\nu(B_l(x)) - \nu(B_k(B_l(x)))} ~~ \text{ (as } B_k(\nu)(\cdot) = \nu(B_k(\cdot)))\\
			&= \abs{\inner{y_1' B_l(x) \Omega_{\rho} }{\Omega_{\rho}}_{\rho} - \inner{y_1' B_k(B_l(x))\Omega_{\rho}}{\Omega_{\rho}}_{\rho}} ~~ \text{ (as } \nu(\cdot) = \langle   y_1' (\cdot)\Omega_{\rho}, \Omega_{\rho} \rangle_{\rho}) \\
			&= \abs{\inner{y_1' B_l(x) \Omega_{\rho} }{\Omega_{\rho}}_{\rho} - \inner{B_k'(y_1') (B_l(x))\Omega_{\rho}}{\Omega_{\rho}}_{\rho}}, \text{ (by  Lemma } \ref{au conv on dense set-1}) \\
			&= \abs{\inner{(y_1'- B_k'(y_1'))B_l(x) \Omega_{\rho}}{\Omega_{\rho}}_{\rho}} \\
			&= \abs{\inner{B_l'(y_1'- B_k'(y_1'))x \Omega_{\rho}}{\Omega_{\rho}}_{\rho}}, \text{ (by  Lemma } \ref{au conv on dense set-1}) \\
			&\leq \norm{B_l'(y_1'- B_k'(y_1')} \rho(x).
		\end{align*}
		
		Further, for all  $n \in \K$, note that
		\begin{align*}
			B_l'(y_1' - B_k'(y_1')) 
			&= \frac{1}{m(I_l)} \int_{I_l} \Big[  T_g'(y_1') -  T_{g}'(B'_k(y_1')) \Big] dm(g) \\
			&= \frac{1}{m(I_l)} \int_{I_l} \Big[  T_{g}'(y_1') -  T_{g}' \Big( \frac{1}{m(I_k)} \int_{I_k}  T_{h}'(y_1') dm(h) \Big) \Big] dm(g) \\
			&= \frac{1}{m(I_l)} \frac{1}{m(I_k)} \int_{I_l} \int_{I_k}  \Big(  T_g'(y_1')-  T'_{hg}(y_1') \Big) dm(h) dm(g) \\
			&= \frac{1}{m(I_k)} \frac{1}{m(I_l)} \int_{I_k} \int_{I_l}  \Big(  T_g'(y_1')-  T'_{hg}(y_1') \Big) dm(g) dm(h)
		\end{align*}
		
		Hence, we have 
		\begin{align*}
			\norm{B_l'(y_1'- B_k'(y_1'))} \leq
			\frac{1}{m(I_k)} \int_{I_k}  \frac{m(I_l \Delta h I_l)}{m(I_l)} dm(h).
		\end{align*}
		
		Now, for all $l \in \K$, consider the function $I_k \ni h \mapsto \frac{m(I_l \Delta h I_l)}{m(I_l)}$. It is a real valued measurable function defined on the compact set $I_k$ and bounded by $2$. Thus, applying DCT we get
		\begin{align}
			\lim_{l \to \infty} \norm{B_l'(y_1'- B_k'(y_1'))} \leq \frac{1}{m(I_k)} \int_{I_k} \lim_{l \to \infty}  \frac{m(I_l \Delta h I_l)}{m(I_l)} dm(h) =0.
		\end{align}
		Hence, we also obtain
		\begin{align*}
			\lim_{l \to \infty} \norm{B_l(\nu_k) -\overline{\nu}}= 0.
		\end{align*}
		
		Therefore, $\bar{\nu}_k= \bar{\nu}$ and we have
		\begin{align*}
			\lim_{l \to \infty} \sup_{x \in M_+, x \neq 0} \abs{(B_l(\nu_k)- \bar{\nu}_k)(x)}/ \rho(x) &=\lim_{l \to \infty} \sup_{x \in M_+, x \neq 0} \abs{(B_l(\nu_k)- \bar{\nu})(x)}/ \rho(x) =0.
		\end{align*}
		This completes the proof.
	\end{proof}

	For the next set of results  we assume that $M$ is a finite von Neumann algebra with a f.n tracial state $\tau$. In due course, we prove the main results in this section which deal with the $b.a.u$ convergence of the ergodic averages in $L^1(M, \tau)$. We start with the following theorem.
	
	\begin{thm}\label{convg theorem in M_*s}
		Let $M$ be a finite von Neumann algebra with f.n trace $\tau$ and $(M, G,  T)$ be a kernel.	Then for any $\mu \in M_{* s}$ there exists an invariant $\overline{\mu} \in M_{*s}$, such that for all $\epsilon>0$, there exists a projection $e \in M$ with $\tau(1-e) < \epsilon$ and
		\begin{align*}
			\lim_{a \to \infty} \sup_{x \in eM_{+}e,~x \neq 0} \abs{\frac{(M_a(\mu) - \bar{\mu})(x)}{\tau(x)}} =0.
		\end{align*}
	\end{thm}
	
	\begin{proof}
		First we note that, since $\rho \in M_{* +}$, there exists a unique $X\in L^1(M,\tau)_+$  such that $\rho(x)= \tau(Xx)$ for all $x \in 	M$. Then for any $s>0$ consider the 	projection $q_{s}:= 	\chi_{(1/s, s)}(X) \in M$.  Observe that $(1-q_{s})\xrightarrow{s \rightarrow \infty} 0$ in SOT and  hence  there exists a $s_0>0$ such that $\tau(1-q_{s_0})< \epsilon/2$. Further, it implies $X q_{s_0} \leq {s_0} q_{s_0}$.  Thus,  for all $0\neq x \in (q_{s_0}M 	q_{s_0})_+$ we have 
		\begin{align}\label{rho-tau}
			\begin{split}
				\frac{\rho(x)}{\tau(x)} = 	\frac{\tau(Xx)}{\tau(x)}&= 
				\frac{\tau(X q_{s_0} x)}{\tau(x)} \quad  \text{ (since $q_{s_0} x= x$)}\\
				&\leq  	\frac{\tau(s_0 x)}{\tau(x)}  = {s_0}.
			\end{split} 
		\end{align}
		
		Now we use Lemma \ref{au conv on dense set 2} recursively to obtain a sequence $\{\nu_{a_1}, \nu_{a_2}, \cdots \} \subseteq \CW$ satisfying 
		\begin{align}\label{coneq-3}
			\nonumber	 (1) ~&a_1 <a_2< \cdots \\
			\nonumber	  (2) ~& \norm{\mu- \nu_{a_j}} < \frac{1}{4^j c} \text{ for  all }  j \in \N \text{ and }\\
			(3) ~& \sup_{x \in M_+, x \neq 0} \frac{\abs{(M_a(\nu_{a_j})- \bar{\nu}_{a_j})(x)}}{ \rho(x)} < \frac{1}{2^j} \text{ for all } a \geq a_j.
		\end{align}
		
		Further, note that  $ \mu - \nu_{a_j}$,  $\bar{\mu} - \bar{\nu}_{a_j} \in M_{* s}$ for all 
		$j \in \N$.  For  every $j \in \N$, take a $N_j \in \N$ and  
		use the  Theorem \ref{Maximal Inequality Real}   to get  sequences of projections $\{e_1, e_2, \ldots  \} $  and $\{f_1, f_2, \ldots \}$ in $M$   such that 
		
		\begin{align}\label{coneq-1}
			\begin{split}
				(1)~&\rho(1-e_j)  < \frac{1}{2^j} \text{ and } \rho(1-f_j)  < \frac{1}{2^j},  \text{ for  all } j\in \N, \\
				(2) ~	&\sup_{x \in e_jM_+e_j, x \neq 0} \frac{\abs{M_a(\mu - \nu_{a_j})(x)}}{\rho(x)} < \frac{1}{2^{j-1}} \text{ for all }  1 \leq a \leq N_j, \\
				(3) ~	& \sup_{x \in f_j M_{+}f_j, x \neq 0} \abs{\frac{M_a(\bar{\mu} - \bar{\nu}_{a_j})(x)}{\rho(x)}} 
				< \frac{1}{2^{j-1}} \text{ for all } 1 \leq a \leq N_j.
			\end{split}
		\end{align}
		
		Now it immediately  follows that both $\rho(1-e_j)$ and $\rho(1-f_j)$ converges to $0$ as $j$ tends to infinity. Therefore,  both $\tau(1-e_j)$ and $\tau(1-f_j)$ 	converges to $0$ as $j$ tends to infinity. Hence choose a subsequence $(j_k)_{k \in \N}$ such that 
		\begin{align}
			\tau(1-e_{j_k}) < \frac{\epsilon}{2^{k+2}} \text{ and } \tau(1-f_{j_k}) < \frac{\epsilon}{2^{k+2}}.
		\end{align}
		Now  consider $e:= \wedge_{k \geq 1}(e_{j_k} \wedge f_{j_k}) \wedge q_{s_0} $ and  observe that
		\begin{align*}
			\tau(1-e)\leq \sum_{k \geq 1} (\tau(1- e_{j_k})+  \tau(1- f_{j_k}))+ \tau(1-q_s) < \epsilon.
		\end{align*}
		
		\noindent Therefore, for  all $0 \neq x \in eM_+e$, and  $a, a_k \in \R_+$, we have
		\begin{align*}
			& \frac{\abs{(M_a(\mu) - \bar{\mu})(x)}}{\tau(x)} \\
			\leq & \Big(\frac{\abs{(M_a(\mu - \nu_{a_k})(x)}}{\rho(x)}  +  \frac{\abs{(\bar{\mu}- \bar{\nu}_{a_k})(x)})}{\rho(x)}  + 
			\frac{\abs{(M_a(\nu_{a_k}) - \bar{\nu}_{a_k})(x)}}{\rho(x)} \Big)\frac{\rho(x)}{\tau (x)}\\
			\leq & s_0\Big( \big(\frac{\abs{M_a(\mu - \nu_{a_k})(x)}}{\rho(x)}  +  \frac{\abs{M_a(\bar{\mu}- \bar{\nu}_{a_k})(x)}}{\rho(x)}\big)  + 
			\frac{\abs{(M_a(\nu_{a_k}) - \bar{\nu}_{a_k})(x)}}{\rho(x)} \Big)\\
			&\text{( by eq. } \ref{rho-tau}, ~~  \frac{\rho(x)}{\tau (x)} \leq s_0 , \text{ as }e \leq q_{s_0} )\\
			\leq & s_0\Big(\big(\frac{\abs{M_a(\mu - \nu_{a_k})(x)}}{\rho(x)}  +  \frac{\abs{M_a(\bar{\mu}- \bar{\nu}_{a_k})(x)}}{\rho(x)}\big) + 
			\sup_{x \in M_+, x \neq 0} \frac{\abs{(M_a(\nu_{a_k}) - \bar{\nu}_{a_k})(x)}}{\rho(x)} \Big).
		\end{align*}
		
		Hence, by taking supremum over the set $eM_+e \setminus \{0\}$ on both sides of the above inequality and applying eq. \ref{coneq-3} and eq. \ref{coneq-1} we obtain
		\begin{align*}
			\lim_{a \to \infty} \sup_{x \in eM_+e, x \neq 0} \frac{\abs{(M_a(\mu) - \bar{\mu})(x)}}{\tau(x)}= 0.
		\end{align*}
		This completes the proof.
	\end{proof}

	\begin{lem}\label{infinite norm of pYp-sa}
		Let $Y \in L^1(M, \tau)$ be a self-adjoint element.  Let $p \in \CP(M)$ and $\delta > 0$ be such that $\abs{\tau(Y pxp)} \leq \delta \tau(pxp)$ for all $x \in M_+ $,    then $pYp \in M$ and $\norm{pYp} \leq \delta$.
	\end{lem}
	
	\begin{thm}\label{bau conv theorem}
		Let $(M, G,  T, \rho)$ be a kernel and $\tau$ be a f.n tracial state  on $M$. Then for all $Y \in L^1(M, \tau)$, there exists $\overline{Y} \in L^1(M, \tau)$ such that $M_a(Y)$ converges to $\overline{Y}$ bilaterally almost uniformly.
	\end{thm}
	
	\begin{proof}
		Let $Y \in L^1(M, \tau)_s$. Then there exists a unique $\mu \in M_{* s}$ such that $\mu(x)= \tau(Yx)$ for all $x \in M$. by Theorem \ref{mean ergodic thm}, we note that $\norm{\cdot}_1 - \lim_{a \to \infty} M_a(\mu)$ exists and denote it by $\bar{\mu} \in M_{* s}$, i.e, 
		\begin{align*}
			\bar{\mu}:= \norm{\cdot}_1 - \lim_{a \to \infty} M_a(\mu).
		\end{align*}
		As, $\bar{\mu} \in M_{* s}$, so there exists a unique $\bar{Y} \in L^1(M, \tau)_s$ such that $\bar{\mu}(x)= \tau(\bar{Y}x)$ for all $x \in M$.  Note that $\bar{\mu}$ is $G$-invariant.  Let $\epsilon, \delta >0$. Then by Theorem \ref{convg theorem in M_*s} there exists a projection $e \in M$ with $\tau(1-e) < \epsilon$ and there exists $N \in \N$ such that
		\begin{align*}
			\sup_{x \in (eMe)_+\setminus \{0 \}} \abs{\frac{(M_a(\mu) - \bar{\mu})(x)}{\tau(x)}} < \delta, \text{ for all } a \geq N.
		\end{align*}
		Therefore, for all $a \geq N$ we have 
		\begin{align*}
			& \abs{(M_a(\mu) - \bar{\mu})(x)} < \delta \tau(x) \text{ for all } x \in (eMe)_+\setminus \{0 \} \\
			\Rightarrow & \abs{\tau((M_a(Y)-\overline{Y}) eye)} < \delta \tau(eye) \text{ for all } y \in M_+\setminus \{0 \} \\
			\Rightarrow & \abs{\tau(e(M_a(Y)-\overline{Y})e y)} < \delta \tau(eye) \text{ for all } y \in M_+\setminus \{0 \}. \\
		\end{align*}
		Hence by Lemma \ref{infinite norm of pYp-sa}, we have $\norm{e(M_a(Y)-\overline{Y})e} \leq \delta$ for all $a \geq N$.  
	\end{proof}

	\begin{thm}\label{bau conv theorem I}
		Let $\tau$ be a f.n tracial state on $M$ and	$(L^1(M, \tau), G,  T)$ be a non-commutative dynamical system. Also assume that there exists a f.n state $\rho$ on $M$ such that $( M, G,  T^*, \rho)$ is a kernel. Then, for all $Y \in L^1(M, \tau)$ there exists $\overline{Y} \in L^1(M, \tau)$ such that 
		$M_a(Y)$ converges  to $\overline{Y}$ in $b.a.u.$
	\end{thm}
	
	\begin{proof}
		First note that $ \hat{ T_g^*} =  T_g $  for all $g \in G$. Consider the anti-action $\beta= (\beta_g)_{g \in G} := ( T_g^*)_{g \in G}$ on $M$. Then the result follows by applying Theorem \ref{bau conv theorem} on the action $\beta$. 
	\end{proof}
	
	\begin{rem}
		Note that if $\text{sup}_{g \in G} \norm{ T_g} \leq 1 $, then $\norm{ T_g^*(1)}\leq 1$. Therefore, $ T_{g}^*(1) \leq 1$. Thus, the condition 	$\text{sup}_{g \in G} \norm{ T_g} \leq 1 $ together with $\rho$ is $G$-invariant can be put in  Theorem \ref{bau conv theorem I} instead of assuming $( M, G,  T^*, \rho)$ is a kernel. However, former conditions are  stronger  than assuming $( M, G,  T^*, \rho)$ is a kernel. 
	\end{rem}
	
	\begin{cor}\label{bau conv theorem II}
		Let $\tau$ be a f.n tracial state on $M$ and	$(L^1(M, \tau), G,  T)$ be a non-commutative dynamical system. Assume that $ T_g(1)=1$ and 	$ T_g^*(1) \leq 1$ for all $ g \in G $.
		Then, for all $Y \in L^1(M, \tau)$ there exists $\overline{Y} \in L^1(M, \tau)$ such that 
		$M_a(Y)$ converges $b.a.u.$ to $\overline{Y}$.
	\end{cor}
	
	\begin{proof}
		Since $\tau$ is finite, observe that $1 \in L^1(M, \tau)$ and it follows that
		\begin{align*}
			\tau( T^*_{g}(x))= \tau(x) \text{ for all } x \in M \text{ and } g \in G.
		\end{align*}
		Hence,   $( M, G,  T^*, \tau)$  becomes a kernel and thus the result follows from Theorem \ref{bau conv theorem I}.
	\end{proof}

	\section{Free group action}
	
	In this section we study convergence of spherical ergodic averages associated a sequence of $w$-continuous maps $\sigma :=\{\sigma_n\}_{n=0}^\infty$  on a von Neumann algebra  $M$. We obtain  mean convergence and   an   auxiliary maximal ergodic inequality   of the   spherical averages associated to $\sigma$. In the end, we assume that $M$ is finite von Neumann algebra and prove b.a.u convergence   of the   spherical averages associated to the maps on the predual of $M$. We also remark that same can be obtained for the free group action. We begin with following definition. 
	\begin{defn}
		
		Let  $\sigma :=\{\sigma_n\}_{n=0}^\infty$ be a sequence of $w$-continuous maps on a von Neumann algebra  $M$. Then $(M , \sigma )$ is called generalized  noncommutative dynamical system if it satisfies the following
		\begin{enumerate}\label{recarence1}
			\item[$A_1$.] $\sigma_1$ is completely positive, $\sigma_1(1) \leq 1$ and $\sigma_n$ is positive.
			\item[$A_2$.] $\sigma_1 \circ \sigma_n = w \sigma_{n+1} + (1-w) \sigma_{n-1}$ for all $n \in \N$ and $\sigma_0(x)=x$, where $1/2 < w < 1$.
		\end{enumerate}
		
	\end{defn}

	We remark that such  examples naturally arises to study ergodic converges of spherical averages of actions of  free group action.  
	Let $G$ be free group with generators $\{a_1, \ldots, a_r, a_r^{-1}, \ldots, a_r^{-1}\}$. Let $\phi: G \to Aut(M)$ denote a group homomorphism. Define,
	\begin{align*}
		\sigma_n := \frac{1}{\abs{W_n}} \sum_{a \in W_n} \phi(a), ~n \in \N,
	\end{align*}
	where, $W_n$ denotes the set of words of length $n$. It is well known that $\sigma_m \circ \sigma_n = \sigma_m \circ \sigma_n$ for all $m,n \in \N$. Furthermore,
	\begin{align*}
		\sigma_1 \circ \sigma_n = w \sigma_{n+1} + (1-w) \sigma_{n-1}, ~ n \in \N,
	\end{align*}
	and $\sigma_0(x)=x$ for all $x \in M$, where $1/2 < w < 1$ is a fixed number.

	Given a  noncommutative dynamical system $(M , \sigma )$, we consider the following average
	\begin{align*}
		S_n(x)= \frac{1}{n+1} \sum_{0}^{n} \sigma_k(x), ~ n\in \N.
	\end{align*}
	Although, it is referred as spherical average in the context of free group action, we freely  use the  same terminology in our context as well. 
	\subsection{Mean convergence}
	In this subsection we prove mean convergence of spherical averages associated to a generalized kernel. 
	
	\begin{defn}
		Consider a sequence of $w$-continuous maps $\sigma :=\{\sigma_n\}_{n=0}^\infty$  and  a f.n state $\rho$ on $M$. Then $(M, \sigma, \rho)$ is call called generalized  kernel if it satisfies the following			
		\begin{enumerate}\label{recarence}
			\item$\sigma_1$ is completely positive, $\sigma_1(1) \leq 1$ and $\sigma_n$ is positive.
			\item $\sigma_1 \circ \sigma_n = w \sigma_{n+1} + (1-w) \sigma_{n-1}$ for all $n \in \N$ and $\sigma_0(x)=x$, where $1/2 < w \leq 1$.
			\item $\rho \circ \sigma_1 = \sigma_1  \text{ and } \rho(y \sigma_1(x)) = \rho(\sigma_1(y) x) \text {for all }   x, y\in M$.
		\end{enumerate}
	\end{defn}
	
	For given a  $(M, \sigma, \rho)$,  there exists a sequence of operators $\{ u_n\} $ on $L^2(M, \rho)$ such that 
	\begin{align*}
		u_n(x \Omega) = \sigma_n(x) \Omega, ~~ x \in M.
	\end{align*}
	Then immediately observe that 
	\begin{enumerate}
		\item $u_n$ is self adjoint and 
		$u_m \circ u_n = u_n \circ u_m$ for all $m,n \in \N$.
		\item There exists $1/2 < w \leq 1$ such that $u_1 \circ u_n = w u_{n+1} + (1-w) u_{n-1}, ~ n \in \N$.
	\end{enumerate}

	\begin{rem}\label{containtment}
		Consider the following set;  
		\begin{align*}
			D_w:= \{z \in \C : \abs{\sqrt{z +4w-4w^2} + \sqrt{z -4w +4w^2}} \leq 2 \sqrt{w} ~ \text{and}~\\
			\abs{\sqrt{z +4w-4w^2} - \sqrt{z -4w +4w^2}} \leq 2 \sqrt{w}\}.
		\end{align*} 
		Then by \cite{Walk:1997}, we note that $ \sigma(u_1)  \subseteq D_w $. 
	\end{rem}
	
	Then we recall the following  mean ergodic type theorem in this context,  the proof follows from Remark \ref{containtment} and \cite{Walk:1997}.   
	\begin{thm}\label{walker-von N ergodic thm}
		Let  $(M, \sigma, \rho)$ be a kernel, then   the ergodic averages  $\frac{1}{n+1} \sum_{0}^{n} u_n$ converges strongly on $L^2(M, \rho)$ to a projection onto the space $\{\eta \in L^2(M, \rho) : v_1 \eta = \eta \}$.
	\end{thm}

	\begin{rem}
		We note that the above situation is automatic for free group action. Indeed, let $\rho$ be a  f.n state on $M$  such that $\rho \circ \phi(a)(x) = \rho(x)$ for all $x \in M$ and  for all $ a \in G$, Then consider the following spherical 
		average  we can define self adjoints  $\{t_a\}_{a \in G}$ on $L^2(M, \rho)$ by
		\begin{align*}
			\sigma_n = \frac{1}{\abs{W_n}} \sum_{a \in W_n} \phi_a .
		\end{align*}
		Then $(M, \sigma, \rho)$ becomes a kernel. 
	\end{rem}

	\begin{thm}\label{mean ergodic thm-free group}
		Let $(M, \sigma, \rho)$  be a kernel. 	For all $\mu \in M_*$, there exists a $\bar{\mu} \in M_*$ such that 
		\begin{align*}
			\bar{\mu}= \norm{\cdot}_1- \lim_{n \to \infty} S_n(\mu),
		\end{align*}
		where, for all $n \in \N$, $\sigma_n(\mu)= \mu \circ \sigma_n$, $\mu \in M_*$.
	\end{thm}
	
	\begin{proof}
		For $n \in \N$, let $ u_n$ be the associated self adjoint operator defined by $ u_n(x\Omega_{\rho}) = \sigma_n(x)\Omega_{\rho}$ for $ x\in M$. consider $T_n = \frac{1}{n+1} \sum_{0}^{n} u_k$.  Let $y_1, y_2 \in M'(\sigma)$ and define $\psi_{y_1, y_2}(x)= \inner{x y_1 \Omega_{\rho}}{ y_2 \Omega_{\rho}}$ for all $x \in M$. Then there exists a  $z \in M$ such that $y_1^*y_2 \Omega_{\rho}= z \Omega_{\rho}$, where $z= J \sigma_{i/2}(y_2^*y_1)J$. Consequently, for $x \in M$
		\begin{align*}
			S_n(\psi_{y_1, y_2})(x) 
			&= \frac{1}{n+1} \sum_{0}^{n} \inner{\sigma_k(x) y_1 \Omega}{y_2 \Omega}\\
			&= \frac{1}{n+1} \sum_{0}^{n} \inner{u_k(x \Omega)}{z \Omega}\\
			&= \frac{1}{n+1} \sum_{0}^{n} \inner{x \Omega}{u_k(z \Omega)}\\
			&= \inner{x \Omega}{\frac{1}{n+1} \sum_{0}^{n} u_k(z \Omega)}\\
			&= \inner{x \Omega}{T_n(z \Omega)}
		\end{align*}
		Therefore, by Theorem \ref{walker-von N ergodic thm} we obtain that for all $x \in M$, $S_n(\psi_{y_1, y_2})(x) \rightarrow \inner{x \Omega}{\eta}$, where $\eta= \lim_{n \to \infty} T_n(z \Omega)$. Observe that $u_1 \eta = \eta$. By induction argument $u_n\eta = \eta$ for all $n \in \N$. 
		
		Now let $ x \in M$. Further write $x$ as $ x_1+ ix_2 $, where $x_1 , x_2\in M_s$ and 
		then by the previous argument, it follows that $ T_n(x\Omega):= T_n(x_1\Omega_{\rho}) + i T_n(x_2\Omega_{\rho})$ converges in $L^2(M, \rho)$.
		
		Now consider $\overline{\psi}_{y_1, y_2} \in M_*$,  defined by 
		\begin{align*}
			\overline{\psi}_{y_1, y_2}(x)= \inner{x \Omega_{\rho}}{\eta }_{\rho}, ~~x \in M.
		\end{align*}
		Then by a  standard argument it follows that $\overline{\psi}_{y_1, y_2} \circ  \sigma_n = \overline{\psi}_{y_1, y_2}$ for all $n \in \N$ and 
		\begin{align*}
			\overline{\psi}_{y_1, y_2} = \norm{\cdot}_1- \lim_{n \to \infty} S_n(\psi_{y_1, y_2}).
		\end{align*}
		Hence the result follows since the set $\{ \psi_{y_1, y_2} : y_1, y_2 \in M'(\sigma) \}$ is total in $M_*$.
	\end{proof}
	
	\subsection{Maximal inequality}
	In this subsection, we obtain an auxiliary maximal ergodic inequality for a generalized noncommutative dynamical system
	$(M, \sigma)$.

	\begin{thm}\label{erg ieq 2-free group}
		Let $E$ be an ordered Banach space. Consider a sequence of positive maps $\{\kappa_n\}_{n=0}^\infty$ on $E$ satisfying 
		\begin{align*}
			&\kappa_1 \circ \kappa_n = w \kappa_{n+1} + (1-w) \kappa_{n-1} ~ and \\
			&\kappa_0(x)=x ~ \text{for all } x \in E.
		\end{align*}
		Then there exists a $C_w>0$ such that for all $x \in E_+$
		\begin{align*}
			\frac{1}{n+1} \sum_{l=0}^{n} \kappa_l (x) \leq C_w \frac{1}{3n+1} \sum_{l=0}^{3n} \kappa_1^l (x).
		\end{align*}
	\end{thm}
	
	\begin{proof}
		Observe that a simple calculation implies that the composition power $\kappa_1^n$ is a convex combination of $\kappa_l$, $0 \leq l \leq n$. That is $\kappa_1^n = \sum_{l=0}^{n} a_n(l) \kappa_l$, where for all $n \in \N$, $\sum_{l=0}^{n} a_n(l) = 1$ and $0 \leq a_n(l) \leq 1$. 
		
		Now a calculation similar to the proof of \cite[Lemma 1]{nevo1994generalization} estimates the coefficients $a_n(l)$ and the results follows.
	\end{proof}
	
	Now consider a $(M, \sigma)$ as described in the beginning of this section. Then we can consider a sequence of positive maps, on $M_*$, again denoting by $\{\sigma_n\}$ by the abuse of notation, which are defined by $\sigma_n(\mu)= \mu \circ \sigma_n$, $\mu \in M_*$. Then for all $n \in \N$ we consider the averaging operators on $M_*$ defined by 
	\begin{align*}
		S_n(\mu)= \frac{1}{n+1} \sum_{k=0}^{n} \sigma_k(\mu).
	\end{align*}
	
	\begin{thm}\label{Maximal Inequality Real-free group}
		
		Let 	 $(M , \sigma )$  be a generalized  noncommutative dynamical system and 
		$\rho$ be a f.n. state on $M$ such that $\rho \circ \sigma_1= \rho$. Further assume that $\mu \in M_{* s}$ and $\epsilon>0$. Then for any $N \in \N$ there exists a projection $e \in M$ such that $\rho(1-e) < \frac{C_w \norm{\mu}}{\epsilon} $ and 
		\begin{align*}
			\abs{S_n(\mu)(x)} \leq \epsilon \rho(x) \text{ for all } x \in (eMe)_+ \text{ and } 1 \leq n \leq N.
		\end{align*}
	\end{thm}
	
	\begin{proof}
		Proof follows immediately from Theorem \ref{maximal lemma} and Theorem \ref{erg ieq 2-free group}.
	\end{proof}

	\subsection{Pointwise convergence}
	In this subsection, we assume $M$ to be finite von Neumann algebra and then  we prove b.a.u convergence of spherical averages associated to a kernel. 
	
	\begin{lem}\label{au conv on dense set-1-free group}
		Let 	 $(M , \sigma )$  be a generalized   noncommutative dynamical system and 
		$\rho$ be a f.n. state on $M$ such that $\rho \circ \sigma_1= \rho$.  Then there exists a sequence of maps $\{\sigma_n'\}_{n \in \N}$ on $M'$ such that 
		\begin{enumerate}
			\item $\langle  \sigma_n'(y')x \Omega_\rho, \Omega_\rho \rangle_{\rho}  = \langle y'  \sigma_n(x) \Omega_\rho, \Omega_\rho \rangle_{\rho} \text{ for all }   x \in M, y' \in M', n \in \N.$
			\item  $\sigma'_m \circ \sigma'_n = \sigma'_n \circ \sigma'_m$ for all $m,n \in \N$.
			\item  $\sigma'_1 \circ \sigma'_n = w \sigma'_{n+1} + (1-w) \sigma'_{n-1}, ~ n \in \N$.
		\end{enumerate}
		Consequently, $(M', \sigma')$ becomes a noncommutative dynamical system. 
	\end{lem}
	
	\begin{proof}
		For $(1)$ let 	 $y' \in M_+'$ and $n \in \N$, consider the linear functional $\nu_{y'}^n : M \rightarrow \C $  defined by
		\begin{align*}
			\nu_{y'}^n(x) = \langle y'  \sigma_n(x) \Omega_\rho, \Omega_\rho \rangle_{\rho}.
		\end{align*}
		For $x \in M_+$, note that
		\begin{align*}
			\nu_{y'}^n(x) &= \langle y'  \sigma_n(x) \Omega_\rho, \Omega_\rho \rangle_{\rho}\\
			& \leq \norm{y'} \langle   \sigma_n(x) \Omega_\rho, \Omega_\rho \rangle_{\rho}\\
			&= \norm{y'} \rho( \sigma_n(x)) \\
			&= \norm{y'} \rho(x) \\
			& = \norm{y'} \langle  x \Omega_\rho, \Omega_\rho \rangle_{\rho}.
		\end{align*}
		Hence, by Lemma \ref{au conv on dense set-0} there exists  a unique  $z' \in M'$ such that $\nu_{y'}^n(x) = \langle  z'x \Omega_\rho, \Omega_\rho \rangle_{\rho}$. Write $ z' =  \sigma_n'(y')$ and then  we have $   \langle y'  \sigma_n(x) \Omega_\rho, \Omega_\rho \rangle_{\rho} = \langle  \sigma_n'(y') x \Omega_\rho, \Omega_\rho \rangle_{\rho}$. Note that 	$(2)$ and $(3)$ follows immediately from $(1)$. 
	\end{proof}

	\begin{prop}\label{double sum conv-free group}
		With the above notation, 	for all $k \in \Z_+$ we have,
		\begin{align*}
			\lim_{n \to \infty} \norm{S_n'(y' - \sigma_k'(y'))} =0.
		\end{align*}
	\end{prop}
	
	\begin{proof}
		
		First, we note that for all $n, k \in \N$ and $y' \in M'$, we have 
		\begin{align}\label{double sum-free group}
			\nonumber	&S_n'(y' - \sigma_{k+1}'(y'))  \\
			& \quad = (1- \frac{1}{w}) S_n' (y' - \sigma_{k-1}'(y')) - \frac{1}{w} \sigma'_1 (S_n' (\sigma'_k(y') - y')) - \frac{1}{w} S_n' (\sigma'_1(y') - y') .
		\end{align}
		Indeed, 	
		For all $n,k \in \N$ we note that 
		\begin{align*}
			S_n'(y' - \sigma_{k+1}'(y')) 
			&= S_n'\Big[ y' - \frac{1}{w} \Big( \sigma'_1 \circ \sigma'_k(y') - (1-w)  \sigma_{k-1}'(y') \Big) \Big] \\
			&= S_n' \Big(y'- \sigma_{k-1}'(y') \Big) -  \frac{1}{w} S_n' \Big( \sigma'_1 \circ \sigma'_k(y') - \sigma_{k-1}'(y') \Big).
		\end{align*}
		Also we have,
		\begin{align*}
			S_n' \Big( \sigma'_1 \circ \sigma'_k(y') - \sigma_{k-1}'(y') \Big) 
			&= S_n' \Big( \sigma'_1 \circ \sigma'_k(y') - \sigma'_1(y') + \sigma'_1(y') - y' + y' - \sigma_{k-1}'(y') \Big) \\
			&= \sigma'_1 (S_n' (\sigma'_k(y') - y')) + S_n' (\sigma'_1(y') - y') + S_n' (y' - \sigma_{k-1}'(y')).
		\end{align*}
		Therefore, combining the above two equality we obtain eq. \ref{double sum-free group}.

		For $k=0$, the result follows immediately. For $k=1$, we obtain
		\begin{align*}
			\norm{S_n'(y' - \sigma_1'(y'))} 
			&=\norm{\frac{1}{n+1} \sum_{l=0}^{n } \sigma_l' (y' - \sigma_1'(y')) } \\
			&=\norm{ \frac{w}{n+1} (y' - \sigma_{n+1}' (y')) + \frac{1-w}{n+1} \sigma_n' (y')  } \\
			&\leq \frac{1+w}{n+1} \norm{y'}
		\end{align*}
		Therefore, the result is true for $k=1$. Also when $k=2$, then we have
		\begin{align*}
			\norm{S_n'(y' - \sigma_2'(y'))}
			&=\norm{\frac{1}{n+1} \sum_{l=0}^{n } \sigma_l' (y' - \sigma_2'(y')) } \\
			&= \frac{1}{w} \norm{ \sigma'_1 (S_n' (y' - \sigma'_1(y'))) + S_n' (y' - \sigma'_1(y') ) } ~(\text{by eq. } \ref{double sum-free group})\\
			&\leq \frac{2}{w} \norm{ S_n' (y' - \sigma'_1(y')) }\\
			&\leq \frac{2(1+w)}{w(n+1)} \norm{y'} ~ \text{(from the case $k=1$)}
		\end{align*}
		Hence the result is true for $k=2$. Now let us assume that the result is true for $\{3, 4, \ldots, k\}$ where $k \in \N$. Now, by eq. \ref{double sum-free group} it follows that 
		\begin{align*}
			&\norm{S_n'(y' - \sigma_{k+1}'(y'))} \\
			&  \leq (1- \frac{1}{w}) \norm{S_n' (y' - \sigma_{k-1}'(y'))} +\frac{1}{w} \norm{\sigma'_1 (S_n' (\sigma'_k(y') - y'))} + \frac{1}{w} \norm{S_n' (\sigma'_1(y') - y')} \\
			& \leq (1- \frac{1}{w}) \norm{S_n' (y' - \sigma_{k-1}'(y'))} +\frac{1}{w} \norm{S_n' (\sigma'_k(y') - y')} + \frac{1}{w} \norm{S_n' (\sigma'_1(y') - y')}.
		\end{align*}
		Hence by induction hypothesis, the result follows for $k+1$. Therefore, the results holds for any $k \in \N$.
	\end{proof}

	\begin{lem}\label{au conv on dense set 2-free group} Let $(M, \sigma, \rho)$ be a generalized  kernel. 
		Consider the following set 
		\begin{align*}
			\CW_1 := \{ \nu - S_k(\nu ) + \bar{\nu} : \quad  k \in \N,\  \nu \in M_{*+} \text{ with } \nu \leq \lambda \rho \text{ for some } \lambda >0 \}.
		\end{align*}
		\begin{enumerate} 
			\item[(i)] Write $\CW= \CW_1-\CW_1$, then $\CW$ is dense in $M_{* s}$ and 
			\item[(ii)] for all $\nu \in \CW$, we have  
			\begin{align}\label{maximal type}
				\lim_{n \to \infty} \sup_{x \in M_+, x \neq 0} \abs{(S_n(\nu)- \bar{\nu})(x)}/ \rho(x) =0.
			\end{align}
		\end{enumerate}
	\end{lem}
	
	\begin{proof}
		\emph{ (i):}	Let $\mu \in M_{* +}$ and $\epsilon>0$. From Theorem \ref{dense subset of normal functionals}, find a $\nu \in M_{*+}$ with $\nu \leq \lambda \rho $ for some $\lambda >0$, such that  $ \norm{\mu - \nu}<\epsilon/2 $. Further, by Theorem \ref{mean ergodic thm-free group}  we know that $S_n(\nu)$ is convergent, and write  $\bar{\nu} = \lim_{n \to \infty} S_n(\nu)$. So there exists a $n_0 \in \N $ such that $\norm{ \bar{\nu} -S_{n_0}(\nu)} \leq \epsilon/2 $. Therefore by triangle inequality, we have 
		$ \norm{\mu-(\nu -S_{n_0}(\nu) + \bar{\nu})} \leq \norm{\mu- \nu} +  \norm{S_{n_0}(\nu) - \bar{\nu}} < \epsilon$.\\
		Now for $\mu \in M_{* s}$, we write $\mu= \mu_+ - \mu_-$, where $\mu_+, \mu_-$ are normal positive linear functional. Thus, it follows that $\CW$ is dense  in $ M_{* s}$.\\
		
		\noindent \emph{ (ii):} Fix $k \in \N$ and consider $ \nu_k := \nu - S_k(\nu ) + \bar{\nu} $ and it is enough to prove eq. \ref{maximal type} for $\nu_k $. First we claim that $\bar{\nu}_k = \bar{\nu}$. 
		
		Since $\nu \leq \lambda \rho $, so by  Lemma \ref{au conv on dense set-0} there exists a unique $y_1' \in M_+'$ with $  y_1' \leq \lambda $ such that 
		\begin{align*}
			\nu(x)= \langle y_1'x \Omega_{\rho}, \Omega_{\rho} \rangle_{\rho} \text{ for all } x \in M.
		\end{align*}
		Let  $y' \in M'$, write $ S_n'(y') := \frac{1}{n+1}\sum_0^n  \sigma_n'(y') $ and by Lemma \ref{au conv on dense set-1-free group}, we have
		\begin{align*}
			\langle S_n'(y') x \Omega_{\rho}, \Omega_{\rho} \rangle_{\rho} = \langle y' S_n(x) \Omega_{\rho}, \Omega_{\rho} \rangle_{\rho}, \quad x \in M,  y' \in M'.
		\end{align*}	
		
		\noindent Now  for all $n \in \N$ and $x \in M_+$, we note that 
		\begin{align*}
			\abs{(S_n(\nu_k) - \bar{\nu})(x)}
			&= \abs{(\nu_k - \bar{\nu})(S_n(x))}\\
			&= \abs{(\nu - S_k(\nu))(S_n(x))} ~~ \text{ (since } \nu_k := \nu - S_k(\nu ) + \bar{\nu})\\
			&= \abs{\nu(S_n(x)) - \nu(S_k(S_n(x)))} ~~ \text{ (as } S_k(\nu)(\cdot) = \nu(S_k(\cdot)))\\
			&= \abs{\inner{y_1' S_n(x) \Omega_{\rho} }{\Omega_{\rho}}_{\rho} - \inner{y_1' S_k(S_n(x))\Omega_{\rho}}{\Omega_{\rho}}_{\rho}} ~~ \text{ (as } \nu(\cdot) = \langle   y_1' (\cdot)\Omega_{\rho}, \Omega_{\rho} \rangle_{\rho}) \\
			&= \abs{\inner{y_1' S_n(x) \Omega_{\rho} }{\Omega_{\rho}}_{\rho} - \inner{S_k'(y_1') (S_n(x))\Omega_{\rho}}{\Omega_{\rho}}_{\rho}}, \text{ (by  Lemma } \ref{au conv on dense set-1-free group}) \\
			&= \abs{\inner{(y_1'- S_k'(y_1'))S_n(x) \Omega_{\rho}}{\Omega_{\rho}}_{\rho}} \\
			&= \abs{\inner{S_n'(y_1'- S_k'(y_1'))x \Omega_{\rho}}{\Omega_{\rho}}_{\rho}}, \text{ (by  Lemma } \ref{au conv on dense set-1-free group}) \\
			&\leq \norm{S_n'(y_1'- S_k'(y_1'))} \rho(x).
		\end{align*}
		Further for $n,k \in \N$ we have
		\begin{align*}
			S_n'(y_1'- S_k'(y_1') )
			&= S_n' \Big[\frac{1}{k+1} \sum_{l=0}^{k} \Big( y_1' - \sigma_l'(y_1')  \Big) \Big]\\
			&= \frac{1}{k+1} \sum_{l=0}^{k} \Big[ S_n' \Big( y_1' - \sigma_l'(y_1')  \Big) \Big].
		\end{align*}
		Now by Proposition \ref{double sum conv-free group} we obtain 
		\begin{align*}
			\lim_{n \to \infty} \norm{S_n(\nu_k) -\overline{\nu}}= 0.
		\end{align*}
		Therefore, $\bar{\nu}_k= \bar{\nu}$ and we have
		\begin{align*}
			\lim_{n \to \infty} \sup_{x \in M_+, x \neq 0} \abs{(S_n(\nu_k)- \bar{\nu}_k)(x)}/ \rho(x) &=\lim_{n \to \infty} \sup_{x \in M_+, x \neq 0} \abs{(S_n(\nu_k)- \bar{\nu})(x)}/ \rho(x) =0.
		\end{align*}
		This completes the proof.
	\end{proof}
	
	Now we assume that $M$ is a finite von Neumann algebra with a f.n tracial state $\tau$ and prove the main results in this section which deal with the $b.a.u$ convergence of the ergodic averages in $L^1(M, \tau)$ for a kernel associated to $ \sigma= \{\sigma_n \}$.
	
	\begin{thm}\label{convg theorem in M_*s-free group}
		Let $M$ be a finite von Neumann algebra with f.n trace $\tau$ and 
		let  $(M, \sigma, \rho)$  be a generalized  kernel. Then we have the following; 	
		\begin{enumerate}
			\item for any $\mu \in M_{* s}$ there exists an invariant $\overline{\mu} \in M_{*s}$, such that for all $\epsilon>0$, there exists a projection $e \in M$ with $\tau(1-e) < \epsilon$ and
			\begin{align*}
				\lim_{n \to \infty} \sup_{x \in eM_{+}e,~x \neq 0} \abs{\frac{(S_n(\mu) - \bar{\mu})(x)}{\tau(x)}} =0 \text{ and }
			\end{align*}
			
			\item   for all $Y \in L^1(M, \tau)$, there exists $\overline{Y} \in L^1(M, \tau)$ such that $S_n(Y)$ converges to $\overline{Y}$ bilaterally almost uniformly.
		\end{enumerate}	
	\end{thm}
	
	\begin{proof}
		The proofs of $(1)$ and $(2)$ follow similarly as Theorem \ref{convg theorem in M_*s} and  Theorem \ref{bau conv theorem}
	\end{proof}

	%

\begin{thebibliography}{HLW21}
		
		\bibitem[AK63]{arnold1963equidistribution}
		VI~Arnold and AL~Krylov, \emph{Equidistribution of points on a sphere and
			ergodic properties of solutions of ordinary differential equations in a
			complex domain}, Dokl. Akad. Nauk SSSR, vol. 148, 1963, pp.~9--12.
		
		\bibitem[Bru73]{brunel73}
		A.~Brunel, \emph{Th\'{e}or\`eme ergodique ponctuel pour un semi-groupe
			commutatif finiment engendr\'{e} de contractions de {$L^{1}$}}, Ann. Inst. H.
		Poincar\'{e} Sect. B (N.S.) \textbf{9} (1973), 327--343. \MR{0344415}
		
		\bibitem[BS23]{Bik-Dip-neveu}
		Panchugopal Bikram and Diptesh Saha, \emph{On the non-commutative neveu
			decomposition and ergodic theorems for amenable group action}, Journal of
		Functional Analysis \textbf{284} (2023), no.~1, 109706.
		
		\bibitem[CDN78]{Conze1978}
		J.-P. Conze and N.~Dang-Ngoc, \emph{Ergodic theorems for noncommutative
			dynamical systems}, Inventiones Mathematicae \textbf{46} (1978), no.~1,
		1--15. \MR{500185}
		
		\bibitem[CLS05]{chilin2005few}
		Vladimir Chilin, Semyon Litvinov, and Adam Skalski, \emph{A few remarks in
			non-commutative ergodic theory}, Journal of Operator Theory (2005), 331--350.
		
		\bibitem[Gri86]{grigorchuk1986individual}
		Rostislav~I Grigorchuk, \emph{Individual ergodic theorem for actions of free
			groups}, Proceedings of the Tambov workshop in the theory of functions, 1986,
		pp.~3--15.
		
		\bibitem[Gui69]{guivarc1969generalisation}
		Yves Guivarc’h, \emph{G{\'e}n{\'e}ralisation d’un th{\'e}oreme de von
			neumann}, CR Acad. Sci. Paris \textbf{268} (1969), 1020--1023.
		
		\bibitem[Hia20]{Hiai2020}
		Fumio Hiai, \emph{Concise lectures on selected topics of von neumann algebras},
		arXiv preprint arXiv:2004.02383 (2020).
		
		\bibitem[HLW21]{Hong2021}
		Guixiang Hong, Ben Liao, and Simeng Wang, \emph{Noncommutative maximal ergodic
			inequalities associated with doubling conditions}, Duke Mathematical Journal
		\textbf{170} (2021), no.~2, 205--246. \MR{4202493}
		
		\bibitem[Hu08]{hu2008maximal}
		Ying Hu, \emph{Maximal ergodic theorems for some group actions}, Journal of
		Functional Analysis \textbf{254} (2008), no.~5, 1282--1306.
		
		\bibitem[JX07]{Junge2007}
		Marius Junge and Quanhua Xu, \emph{Noncommutative maximal ergodic theorems},
		Journal of the American Mathematical Society \textbf{20} (2007), no.~2,
		385--439. \MR{2276775}
		
		\bibitem[K\"78]{Kuemmerer1978}
		Burkhard K\"{u}mmerer, \emph{A non-commutative individual ergodic theorem},
		Inventiones Mathematicae \textbf{46} (1978), no.~2, 139--145. \MR{482260}
		
		\bibitem[Kad52]{Kadison:1952vc}
		Richard~V. Kadison, \emph{A generalized {S}chwarz inequality and algebraic
			invariants for operator algebras}, Ann. of Math. (2) \textbf{56} (1952),
		494--503. \MR{51442}
		
		\bibitem[Kre85]{Krengel1985}
		Ulrich Krengel, \emph{Ergodic theorems}, De Gruyter Studies in Mathematics,
		vol.~6, Walter de Gruyter \& Co., Berlin, 1985, With a supplement by Antoine
		Brunel. \MR{797411}
		
		\bibitem[Lan76]{Lance1976}
		E.~Christopher Lance, \emph{Ergodic theorems for convex sets and operator
			algebras}, Inventiones Mathematicae \textbf{37} (1976), no.~3, 201--214.
		\MR{428060}
		
		\bibitem[Nev94]{nevo1994harmonic}
		Amos Nevo, \emph{Harmonic analysis and pointwise ergodic theorems for
			noncommuting transformations}, Journal of the American Mathematical Society
		\textbf{7} (1994), no.~4, 875--902.
		
		\bibitem[NS94]{nevo1994generalization}
		Amos Nevo and Elias~M Stein, \emph{A generalization of birkhoff's pointwise
			ergodic theorem}, Acta Mathematica \textbf{173} (1994), no.~1, 135--154.
		
		\bibitem[Pau02]{Paulsen:2002tj}
		Vern Paulsen, \emph{Completely bounded maps and operator algebras}, Cambridge
		Studies in Advanced Mathematics, vol.~78, Cambridge University Press,
		Cambridge, 2002. \MR{1976867}
		
		\bibitem[SZ79]{stratila79}
		\c{S}erban Str\u{a}til\u{a} and L\'{a}szl\'{o} Zsid\'{o}, \emph{Lectures on von
			{N}eumann algebras}, Editura Academiei, Bucharest; Abacus Press, Tunbridge
		Wells, 1979, Revision of the 1975 original, Translated from the Romanian by
		Silviu Teleman. \MR{526399}
		
		\bibitem[Wal97]{Walk:1997}
		Trent~E. Walker, \emph{Ergodic theorems for free group actions on von {N}eumann
			algebras}, J. Funct. Anal. \textbf{150} (1997), no.~1, 27--47. \MR{1473624}
		
		\bibitem[Yea77]{Yeadon1977}
		F.~J. Yeadon, \emph{Ergodic theorems for semifinite von {N}eumann algebras.
			{I}}, Journal of the London Mathematical Society. Second Series \textbf{16}
		(1977), no.~2, 326--332. \MR{487482}
		
		\bibitem[Yea80]{Yeadon1980}
		\bysame, \emph{Ergodic theorems for semifinite von {N}eumann algebras. {II}},
		Mathematical Proceedings of the Cambridge Philosophical Society \textbf{88}
		(1980), no.~1, 135--147. \MR{569639}
		
	\end{thebibliography}
	
	\providecommand{\bysame}{\leavevmode\hbox to3em{\hrulefill}\thinspace}
	\providecommand{\MR}{\relax\ifhmode\unskip\space\fi MR }
	\providecommand{\MRhref}[2]{%
		\href{http://www.ams.org/mathscinet-getitem?mr=#1}{#2}
	}
	\providecommand{\href}[2]{#2}

\end{document}